\documentclass[final,leqno]{siamltex}

\usepackage{amsmath,amsfonts}
\usepackage{enumerate}
\usepackage{caption}
\usepackage{subcaption}
\usepackage{graphicx}
\usepackage{epstopdf}
\usepackage{bbm}
\usepackage{float}
\usepackage{calligra}
\usepackage[ruled,vlined]{algorithm2e}
\usepackage{xparse}
\newtheorem{remark}[theorem]{Remark}

\title{On a Problem of Weighted Low Rank Approximation of Matrices}

\author{Aritra Dutta\thanks{Department of Mathematics, University of Central Florida, 4393 Andromeda Loop N,
		Orlando, FL 32816~({d.aritra2010@knights.ucf.edu}).}
        \and Xin Li\thanks{Department of Mathematics, University of Central Florida, 4393 Andromeda Loop N,
        	Orlando, FL 32816~({xin.li@ucf.edu}).}}

\begin{document}

\maketitle

\begin{abstract}
We study a weighted low rank approximation that is inspired by a problem of constrained low rank approximation of matrices as initiated by the work of Golub, Hoffman, and Stewart (Linear Algebra and Its Applications, 88-89(1987), 317-327).~Our results reduce to that of Golub, Hoffman, and Stewart in the limiting cases. We also propose an algorithm based on the alternating direction method to solve our weighted low rank approximation problem and compare it with the state-of-art general algorithms such as the weighted total alternating least squares and the EM algorithm.
\end{abstract}



\pagestyle{myheadings}
\thispagestyle{plain}
\markboth{ARITRA DUTTA AND XIN LI}{Weighted low rank approximation of matrices}

\section{Introduction}
Let $m$ and $n$ be two natural numbers.~For an integer $r\le\min\{m,n\}$ and a matrix $A\in\mathbb{R}^{m\times n}$, the standard low rank approximation problem can be formulated as
\begin{eqnarray}
\label{pca}
\min_{\substack{{X}\in\mathbb{R}^{m\times n}\\{\rm r}({X})\le r}}\|A-{X}\|_F^2,
\end{eqnarray}
where ${\rm r}({X})$ denotes the rank of the matrix ${X}$ and $\|\cdot\|_F$ denotes the Frobenius norm of matrices.

It is well known that the solutions to this problem can be given using the singular value decompositions~(SVDs) of $A$ through the hard thresholding operations on the singular values: The solutions to~(\ref{pca}) are given by
\begin{align}\label{hardthresholding}
X^*=H_r(A):=U(A)\Sigma_r(A)V(A)^T,
\end{align}
where
$
A=U(A)\Sigma(A)V(A)^T
$
is a SVD of $A$ and $\Sigma_r(A)$ is the diagonal matrix obtained from $\Sigma(A)$ by thresholding:~keeping only $r$  largest singular values and replacing other singular values by 0 along the diagonal.~This is also referred to as Eckart-Young-Mirsky's theorem~(\cite{svd}) and is closely related to the PCA method in statistics~\cite{pca}. The solutions to~(\ref{pca}) as given in~(\ref{hardthresholding}) suffer from the fact that none of the entries of $X$ is guaranteed to be preserved in $X^*$. In many real world problems, one has good reasons to keep certain entries of $A$ unchanged while looking for a low rank approximation.
In 1987, Golub, Hoffman, and Stewart were the first to consider the following {\it constrained} low rank approximation problem~\cite{golub}:

Given $A=(A_1\;\;A_2)\in\mathbb{R}^{m\times n}$ with $A_1\in\mathbb{R}^{m\times k}$ and $A_2\in\mathbb{R}^{m\times (n-k)}$, find $\tilde{A}_2$ such that~(with $\tilde{A}_1=A_1$)
\begin{eqnarray}
(\tilde{A}_1\;\tilde{A}_2)=\arg\min_{\substack{X_1,X_2\\{\rm r}(X_1\;\;X_2)\le r\\X_1=A_1}}\|(A_1\;\;A_2)-(X_1\;\;{X}_2)\|_F^2.\label{golub's problem}
\end{eqnarray}
That is, Golub, Hoffman, and Stewart required that certain columns, $A_1,$ of $A$ must be preserved when one looks for a low rank approximation of $(A_1\;\;A_2).$ As in the standard low rank approximation, the constrained low rank approximation problem of Golub, Hoffman, and Stewart also has a closed form solution.

\begin{theorem}\cite{golub}
\label{theorem 1}
With $k={\rm r}(A_1)$ and $r\ge k$, the solutions $\tilde{A}_2$ in~(\ref{golub's problem}) are given by
	\begin{align}\label{ghs}
	\tilde{A_2}= P_{A_1}(A_2)+H_{r-k}\left(P^{\perp}_{A_1}(A_2)\right),
	\end{align}
where $P_{A_1}$ and $P^\perp_{A_1}$ are the projection operators to the column space of $A_1$ and its orthogonal complement, respectively.
\end{theorem}

%Let us define $A_G=(A_1\;\;\tilde{A}_2).$
\begin{remark}\label{remark1}
	{\rm According to Section 3 of~\cite{golub}, the matrix $\tilde{A}_2$ is unique if and only if $H_{r-k}\left(P^{\perp}_{A_1}(A_2)\right)$ is unique, which means the $(r-k)$th singular value of $P^{\perp}_{A_1}(A_2)$ is strictly greater than $(r-k+1)$th singular value. When $\tilde{A}_2$ is not unique, the formula for $\tilde{A}_2$ given in Theorem~\ref{theorem 1} should be understood as the membership of the set specified by the right-hand side of~(\ref{ghs}). We will use this convention in this paper.}
\end{remark}

Inspired by Theorem~\ref{theorem 1} above and motivated by applications in which $A_1$ may contain noise, it makes more sense if we require $\|A_1-X_1\|_F$ small instead of asking for $X_1=A_1$. This leads us to consider the following problem: Let $\eta>0$, find $(\hat{X}_1\;\;\hat{X}_2)$ such that
\begin{eqnarray}\label{closeness problem}
(\hat{X}_1\;\;\hat{X}_2)=\arg\min_{\substack{X_1,X_2: \|A_1-X_1\|_F\leq \eta\\{\rm r}(X_1\;\;X_2)\le r}}\| (A_1\;\;A_2)-({X}_1\;\;{X}_2) \|_F^2.
\end{eqnarray}
Or, for a large parameter $\lambda $, consider
\begin{eqnarray}\label{unconstraint closeness}
(\hat{X}_1\;\;\hat{X}_2)=\arg\min_{\substack{X_1,X_2\\{\rm r}(X_1\;\;X_2)\le r}}\left\{ \lambda^2\|A_1-X_1\|_F^2+\| A_2-{X}_2 \|_F^2\right\}.
\end{eqnarray}
This block structure in weight matrix, where very few entries are heavily weighted and most entries stays at 1 (unweighted), is realistic in applications. For example, in the problem of background estimation in a video sequence, each frame is a column in the data matrix and the background is a low rank (ideally of rank 1) component of the data matrix. So, the weight is used to single out which columns are more likely to be the basis of background frames and the low rank constraint enforces the search for other frames that are in the background subspace.
Recent investigations in~\cite{duttaligongshah,dutta_li,dutta_li_acl} have shown that the above ``approximately'' preserving (controlled by a parameter $\lambda$) weighted low rank approximation can be more effective in solving the background modeling, shadows and specularities removal, and domain adaptation problems in computer vision and machine learning.
%some elements of the data matrix

As it turns out, (\ref{unconstraint closeness}) can be viewed as a generalized total least squares problem~(GTLS) and can be solved in closed form as a special case of weighted low rank approximation with a rank-one weight matrix by using a SVD of the given matrix $(\lambda A_1\;\;A_2)$~\cite{markovosky,markovosky1}. As a consequence of the closed form solutions, one can verify that the solution to (\ref{golub's problem}) is the limit case of the solutions to (\ref{unconstraint closeness}) as $\lambda \to\infty$. Thus, (\ref{golub's problem}) can be viewed as a special case when ``$\lambda = \infty$''. A careful reader may also note that, problem (\ref{unconstraint closeness}) can be cast as a special case of structured low rank problems with element-wise weights~\cite{markovosky3,markovosky4}.
%But what about using a general weight $W$ as defined in~(\ref{wlr}), instead of using the special weight $W_{\lambda}$?
More specifically, we observe that (\ref{unconstraint closeness}) is contained in the following more general point-wise weighted low rank approximation problem:
\begin{eqnarray}\label{hadamard problem}
	\min_{\substack{X_1,X_2\\{\rm r}(X_1\;\;X_2)\le r}}\|\left((A_1\;\;A_2)-({X}_1\;\;{X}_2)\right)\odot(W_1\;~W_2)\|_F^2,
\end{eqnarray}
for $W_1=\lambda \mathbbm{1}$ and $W_2=\mathbbm{1}$ (a matrix whose entries are equal to 1), where $W\in\mathbb{R}^{m\times n}$ is a weight matrix and $\odot$ denotes the Hadamard product.

This is the weighted low rank approximation problem studied first when $W$ is an indicator weight for dealing with the missing data case (\cite{wibergjapan,wiberg}) and then for more general weight in machine learning, collaborative filtering, 2-D filter design, and computer vision~\cite{srebro,srebromaxmatrix,Buchanan,manton,lupeiwang,shpak}.
For example, if SVD is used in quadrantally-symmetric two-dimensional~(2-D) filter design, as explained in~\cite{manton} (see also \cite{lupeiwang,shpak}), it might lead to a degraded construction in some cases as it is not able to discriminate between the important and unimportant components of $X$. To address this problem, a weighted least squares matrix decomposition method was first proposed by Shpak~\cite{shpak}. Following his idea of assigning different weights to discriminate between important and unimportant components of the test matrix, Lu, Pei, and Wang~(\cite{lupeiwang}) designed a numerical procedure to solve (\ref{hadamard problem}) for general weight $(W_1\;~W_2)$.

\begin{remark}\label{remark2}
{\rm There is another formulation of weighted low rank approximation problem defined as in~\cite{manton}:
\begin{align}
\label{manton}
\min_{\substack{{X}\in\mathbb{R}^{m\times n}}}\|A-X\|_{Q}^2,~~{\rm subject~to~}{\rm r}({X})\le r,
\end{align}
where $Q \in \mathbb{R}^{mn\times mn}$ is a symmetric positive definite weight matrix.~Denote $\|A-X\|_Q^2:={\rm vec}(A-X)^TQ{\rm vec}(A-X)$, where ${\rm vec}(\cdot)$ is an operator which maps the entries of $\mathbb{R}^{m\times n}$ to $\mathbb{R}^{mn\times 1}$. It is easy to see that (\ref{hadamard problem}) is a special case of~(\ref{manton}) with a diagonal $Q$. In this paper, we will not use this more general formulation for simplicity.
}
\end{remark}

Motivated by the limit behavior of (\ref{unconstraint closeness}) as $\lambda\to\infty$, we are interested in finding out the limit behavior of the solutions to problem~(\ref{hadamard problem}) for general weight $(W_1\;~W_2)$  when $(W_1)_{ij}\to\infty$ and $W_2\to \mathbbm{1}$. One can expect that, with appropriate conditions, the limit should be the solutions to (\ref{golub's problem}). We will verify this with an stronger result with estimation on the rate of convergence~(when $W_2=\mathbbm{1}$). The main challenge here is the lack of closed form solutions to problem (\ref{hadamard problem}) in general \cite{srebro,manton}.~We will also extend the convergence result to the unconstrained version of the problem~(\ref{hadamard problem}). 

In order to make use of the proposed weighted low rank approximation in applications, we will propose a numerical algorithm to solve~(\ref{hadamard problem}) for the special case when $W_2=\mathbbm{1}$. In view of the existing algorithms for solving the general weighted low rank approximation problem, we want to emphasize that our special algorithm takes advantage of the block structure of our weights that results in better numerical performance as well as the fact that we have detailed convergence analysis for the algorithm. 

The rest of the paper is organized as follows.~In Section 2, we state our main results on the behavior of the solutions to~(\ref{hadamard problem}) as $(W_1)_{i,j}\to\infty$ and $(W_2)_{i,j}\to 1$. Their proofs are given in Section 3.~In Section 4, we present a numerical solution to problem~(\ref{hadamard problem}) when $W_2=\mathbbm{1}$ and discuss the convergence of our algorithm.~Numerical results verifying our main results are given in Section 5.

The extension of our rate of convergence and numerical algorithm to the more general case when $W_2=\mathbbm{1}$ is replaced by $W_2\to \mathbbm{1}$ is non-trivial and remains open.

\section{Limiting behavior of solutions as $(W_1)_{i,j}\to\infty$ and $(W_2)_{i,j}\to 1$}
Denote $\mathcal{A}=P_{A_1}^\perp(A_2)$ and $\tilde{\mathcal{A}}=P_{\tilde{X}_1(W)}^\perp(A_2).$ Also denote $s={\rm r}(\mathcal{A})$ and let the ordered non-zero singular values of $\mathcal{A}$ be  $\sigma_1\ge\sigma_2\ge\cdots\ge\sigma_{s}>0.$
Let $(\tilde{X}_1(W)\;~\tilde{X}_2(W))$ be a solution to~(\ref{hadamard problem}). % Denote the solutions to (\ref{golub's problem}) by $A_G$.
Let $\lambda_j = \displaystyle{\min_{1\le i\le m}(W_1)_{ij}}$ and $\lambda = \displaystyle{\min_{1\le j\le k}\lambda_j}$.

\begin{theorem}\label{theorem 3} Suppose that $\sigma_{r-k}>\sigma_{r-k+1}$. Then
we have, as $\lambda\to\infty$ and $W_2=\mathbbm{1}$, %%Let $W_2=\mathbbm{1}_{m\times (n-k)}$.
 	$$
(\tilde{X}_1(W)\;\;\tilde{X}_2(W))=A_G+\displaystyle{O(\frac{1}{\lambda})},
	$$
	where $A_G=(A_1\;\tilde{A}_2)$ is defined to be the unique solution to~(\ref{golub's problem}).
\end{theorem}
\begin{remark}
{\rm (i) The assertion of the uniqueness of $A_G$ is due to the assumption $\sigma_{r-k}>\sigma_{r-k+1}$ (see the Remark~\ref{remark1}). (ii) The convergence $(\tilde{X}_1(W)\;\;\tilde{X}_2(W))\to A_G$ alone is indeed implied by a general result in \cite{markovosky3}.}
\end{remark}

Next, if we do not know $r$ but still want to reduce the rank in our approximation, then we can consider the unconstrained version of~(\ref{hadamard problem}):~for~$\tau>0$,
\begin{eqnarray}\label{unconstrained hadamard problem}
\min_{\substack{X_1,X_2}}\left\{\|\left((A_1\;\;A_2)-({X}_1\;\;{X}_2)\right)\odot(W_1\;\;W_2)\|_F^2+\tau{\rm r}({X}_1\;\;{X}_2)\right\}.
\end{eqnarray}
Again, one can expect that the solutions to~(\ref{unconstrained hadamard problem}) will converge to the solution of (\ref{golub's problem}) as $(W_1)_{ij}\to\infty$ and $(W_2)_{ij}\to 1$. We will first establish a convergence result for~(\ref{unconstrained hadamard problem}) without assuming the uniqueness of the solutions to (\ref{golub's problem}).

Let ${\cal A}_G^r$ be the set of all solutions to~(\ref{golub's problem}) and let $(\hat{X}_1(W)\;~\hat{X}_2(W))$ be a solution to~(\ref{unconstrained hadamard problem}).

\begin{theorem}\label{theorem 5}
Every accumulation point of $(\hat{X}_1(W)\;~\hat{X}_2(W))$ as $(W_1)_{ij}\to\infty,(W_2)_{ij}\to 1$ belongs to $\displaystyle{\mathop{\cup}_{0\le r\le \min\{m,n\}}{\cal A}_G^r}.$
\end{theorem}

Next, we have more precise information of the convergence if we assume the uniqueness of (\ref{golub's problem}).
\begin{theorem}\label{theorem 6}
Assume that $\sigma_1>\sigma_2>\cdots>\sigma_s>0$. Denote $\sigma_0:=\infty$ and $\sigma_{s+1}:=0.$
Then the accumulation point of the sequence $(\hat{X}_1(W)\;\hat{X}_2(W)),$ as $(W_1)_{ij}\to\infty$ and $(W_2)_{ij}\to1$ is unique; and this unique accumulation point is given by
$$
\left(A_1\;\;\;\;P_{A_1}(A_2)+H_{r^*}\left(P^{\perp}_{A_1}(A_2)\right)\right)
$$
with $r^*$ satisfying
	$$
	\sigma_{r^*+1}^2\le\tau<\sigma_{r^*}^2.
	$$
\end{theorem}

\begin{remark}
{\rm For the case when $P^{\perp}_{A_1}(A_2)$ has repeated singular values, we leave it to the reader to verify the following more general statement by using a similar argument: Let $\hat{\sigma}_1>\hat{\sigma}_2>...>\hat{\sigma}_t>0$ be the singular values of $P^{\perp}_{A_1}(A_2)$ with multiplicity $k_1,k_2,\cdots k_t$ respectively. Note that $\sum_{i=1}^tk_i=s.$
Let $\sigma_{p^*+1}^2\le \tau <\sigma_{p^*}^2,$ where $\sigma_{p^*}$ has multiplicity $k_{p^*}.$
Then the accumulation points of the set $(\hat{X}_1(W)\;~\hat{X}_2(W)),$ as $(W_1)_{ij}\to\infty, (W_2)_{ij}\to1$, belongs to the set $\displaystyle{\mathop{\cup}_{r^*}{\cal A}_G^{r^*}},$ where $1+\sum_{i=1}^{p^*-1}k_i\le r^{*}<\sum_{i=1}^{p^*}k_i.$}
\end{remark}

\section{Proofs of results in Section 2}
To prove Theorem~\ref{theorem 3}, we start with a few well known results from the perturbation theory of singular values.~First, we quote the following result of Stewart.

\begin{lemma}\label{lemma 3}\cite{stewart}
	Let $\tilde{{A}}={{A}}+E$ and $\sigma\neq 0$ be a non-repeating singular value of the matrix ${A}$ with $u$ and $v$ being left and right singular vectors respectively. Then as $\|E\| \to 0,$ there is a unique singular value $\tilde{\sigma}$ of $\tilde{{A}}$ such that
	\begin{eqnarray}\label{per-singularvalue}
	\tilde{\sigma}=\sigma+u^TEv+O(\|E\|^2).
	\end{eqnarray}
\end{lemma}

Let the SVDs of ${B},~\tilde{{B}}\in \mathbb{R}^{m\times n}$ be given by
	\begin{eqnarray}\label{SVDA}
	{B}=U\Sigma V^T=(U_1\;U_2)\begin{pmatrix}
	\Sigma_1 & 0\\
	0 &\Sigma_2\\
	\end{pmatrix}\begin{pmatrix}
	V_1^T\\V_2^T
	\end{pmatrix}=:{B}_1+{B}_2,
	\end{eqnarray}
	\begin{eqnarray}\label{SVDA1}
	\tilde{{B}}=\tilde{U}\tilde{\Sigma}\tilde{V}^T=(\tilde{U}_1\;\tilde{U}_2)\begin{pmatrix}
	\tilde{\Sigma}_1 & 0\\
	0 &\tilde{\Sigma}_2\\
	\end{pmatrix}\begin{pmatrix}
	\tilde{V}_1^T\\\tilde{V}_2^T
	\end{pmatrix}=:\tilde{{B}}_1+\tilde{{B}}_2,
	\end{eqnarray}
	such that $U,~\tilde{U}\in\mathbb{R}^{m\times m},~V,~\tilde{V}\in\mathbb{R}^{n\times n},\;\text{and}\;~ \Sigma,~\tilde{\Sigma}\in\mathbb{R}^{m\times n}$ with $\Sigma$ and $\tilde{\Sigma}$ being diagonal matrices containing singular values of ${B}$ and $\tilde{{B}}$, respectively, arranged in a non-increasing order; $U_1,~\tilde{U}_1\in\mathbb{R}^{m\times s},U_2,~\tilde{U}_2\in\mathbb{R}^{m\times (m-s)},
	V_1,~\tilde{V}_1\in\mathbb{R}^{n\times s},\;\text{and}\;~
	V_2,~\tilde{V}_2\in\mathbb{R}^{n\times (n-s)}.$
	Using~(\ref{SVDA}) and~(\ref{SVDA1}) we have, with $E=\tilde{B}-B$,
	\begin{eqnarray}\label{per_SVD2}
	\tilde{{B}}_1+\tilde{{B}}_2=\tilde{{B}}={B}+E={B}_1+{B}_2+E.
	\end{eqnarray}

Next, we state a special form of the sin$\theta$ Theorem of Wedin (see \cite[p. 260]{stewart-sun} and \cite{wedin}) as follows.
\begin{lemma}\label{Wedin theorem}
With the notations above, let $\eta=\sigma_{min}(\tilde{B}_1)-\sigma_{max}({B}_2)>0$.
Then
$$
\sqrt{\|\sin \Theta(U_1,\tilde{U}_1)\|^2_F+\|\sin \Theta(V_1,\tilde{V}_1)\|^2_F}\leq\frac{\sqrt{2}\|E\|_F}{\eta}.
$$
\end{lemma}

Finally, using the argument of Wedin~(\cite{wedin}), the following result can be achieved.
\begin{lemma}\label{lemma 4} (\cite[Sect. 4.4]{wedin}).
	Assume there exists an $\alpha\ge 0$ and a $\delta>0$ such that
	$$
	\sigma_{min}(\tilde{{B}}_1)\ge \alpha+\delta\;\;\text{and}\;\;\sigma_{max}({B}_2)\le\alpha,
	$$
	then
	\begin{align}\label{thresholding inequality}
	\|{B}_1-\tilde{{B}}_1\|\le \|E\|(3+\frac{\|{B}_2\|}{\delta}+\frac{\|\tilde{{B}}_2\|}{\delta}).\end{align}
\end{lemma}

Now, we establish some auxiliary results. Let $P_B$ and $P_{\tilde{B}}$ be the orthogonal projections onto the column spaces of matrices $B$ and $\tilde{B}$, respectively.
\begin{lemma}\label{projections} Assume that $B$ and $\tilde{B}$ are full rank matrices. Let $\eta$ denote the smallest singular value of $\tilde{B}$. Then
\begin{equation}\label{project estimation}
\|P_B-P_{\tilde{B}}\|_F\leq \frac{2\|B-\tilde{B}\|_F}{\eta}.
\end{equation}
\end{lemma}

\begin{proof}
According to \cite[p. 43]{stewart-sun}, we have
\begin{equation}\label{x1}
\|P_B-P_{\tilde{B}}\|_F=\sqrt{2}\|\sin \Theta (B,\tilde{B})\|_F,
\end{equation}
where $\Theta (B,\tilde{B})$ is the canonical angle (diagonal) matrix between the column spaces of $B$ and $\tilde{B}$. Now, by applying Lemma~\ref{Wedin theorem} to the case when $U_2=\tilde{U}_2=\emptyset$ (so that $U_1$ and $\tilde{U}_1$ span the same column spaces as those of $B$ and $\tilde{B}$, respectively, and so $\Theta(B,\tilde{B})=\Theta(U_1,\tilde{U}_1)$ with $\eta=\sigma_{min}(\tilde{B}_1)>0$ (since $\sigma({B}_2)=0$), we have
\begin{equation}\label{wedin1}
\|\sin \Theta (B, \tilde{B})\|_F=\|\sin \Theta (U_1, \tilde{U}_1)\|_F\leq \frac{\sqrt{2}\|B-\tilde{B}\|_F}{\eta}.
\end{equation}
Now, the inequality (\ref{project estimation}) follows from (\ref{x1}) and (\ref{wedin1}).
\end{proof}

\noindent\begin{lemma}\label{lemma 1}
As $(W_1)_{ij}\to\infty$ and $W_2$ stays bounded, we have the following estimates.
\begin{enumerate}[(i)]
		\item$\tilde{X}_1(W)= A_1+\displaystyle{O(\frac{1}{\lambda})}$.
		\item $P_{\tilde{X}_1(W)}(A_2)=P_{A_1}(A_2)+\displaystyle{O(\frac{1}{\lambda})}$.
		\item $P_{\tilde{X}_1(W)}^\perp(A_2)=P_{A_1}^\perp(A_2)+\displaystyle{O(\frac{1}{\lambda})}$.
\end{enumerate}
\end{lemma}
	\noindent\begin{proof}{\it (i).}
	Note that,
	\begin{align*}
	&\|(A_1-\tilde{X}_1(W))\odot W_1\|_F^2+\|(A_2-\tilde{X}_2(W))\odot W_2\|_F^2\\
	&=\min_{\substack{{X}_1,{X}_2\\{\rm r}({X}_1\;\;{X}_2)\le r}}\left(\|(A_1-{X}_1)\odot W_1\|_F^2+\|(A_2-{X}_2)\odot W_2\|_F^2\right)\\
	&\le\|A_2\odot W_2\|_F^2\;(\text{by taking}\;(X_1\;\;X_2)=(A_1\;\;0)).
	\end{align*}
	Then $\displaystyle{\sum_{\substack{1\le i\le m\\1\le j\le k}}((A_1)_{ij}-(\tilde{X}_1(W))_{ij})^2(W_1)_{ij}^2\le \|A_2\odot W_2\|_F^2}$ and so
	$$|(A_1)_{ij}-(\tilde{X}_1(W))_{ij}|\le\frac{\|A_2\odot W_2\|_F}{(W_1)_{ij}};~~1\le i\le m, 1\le j\le k.$$ Thus
	$$ \tilde{X}_1(W)= A_1+\displaystyle{O(\frac{1}{\lambda})}\;\text{as}\;\lambda\to\infty.$$
	\noindent{\it (ii).} Note that as $\lambda\to \infty$, $\tilde{X}_1(W)\to A_1$. So, $\tilde{X}_1(W)$ will be a full rank matrix (since $A_1$ is assumed to be of full rank).
	Applying (i) and Lemma~\ref{project estimation} to $B=A_1$ and $\tilde{B}=\tilde{X}_1(W)$, we have, as $\lambda\to \infty$,
	$$
	\|P_{A_1}-P_{\tilde{X}_1(W)}\|_F \leq \frac{2}{\eta (\tilde{X}_1(W))} O(\frac{1}{\lambda}).
	$$
	Since 	$\eta (\tilde{X}_1(W))\to \eta(A_1)>0$, we have $\eta (\tilde{X}_1(W))\geq \eta(A_1)/2$ as
	$\lambda\to \infty$. Thus, (ii) holds.
\\	
	{\it (iii)} We know that
	$$
	P_{\tilde{X}_1(W)}(A_2)+P_{\tilde{X}_1(W)}^\perp(A_2)=A_2=P_{A_1}(A_2)+P_{A_1}^{\perp}(A_2).
	$$
	Using (ii)
	\begin{align*}
	& P_{A_1}(A_2)+\displaystyle{O(\frac{1}{\lambda})}+P_{\tilde{X}_1(W)}^\perp(A_2)=P_{A_1}(A_2)+P_{A_1}^{\perp}(A_2),~\lambda\to\infty.
	\end{align*}
	Therefore,
	\begin{eqnarray}\label{perp}
	& P_{\tilde{X}_1(W)}^\perp(A_2)=P_{A_1}^\perp(A_2)+O(\frac{1}{\lambda}),~\lambda\to\infty.
	\end{eqnarray}
	This completes the proof of Lemma~\ref{lemma 1}.\qquad\end{proof}
\begin{remark}
	{\rm (i) For the case when there is an uniform weight in $(W_1)_{ij}=\lambda>0$, one might refer to~\cite{stewart2} for an alternative proof of Lemma~\ref{lemma 1}. But the proof in~\cite{stewart2} can not be applied in the more general weight. (ii) There is a more elementary proof of (ii)~of Lemma~\ref{lemma 1} above by using the Gram-Schmidt process~(see~\cite{dutta_thesis}), instead of using the advanced tools from the perturbation theory of singular values of Stewart and Wedin. We choose to use the shorter proof here because we need to invoke the advanced theory for the proof of the next lemma anyway.}
\end{remark}

Next, we establish a key lemma on the hard thresholding operator under perturbation. We use the notation introduced at the beginning of Section 2.
\begin{lemma}\label{lemma 2}
	Let $W_2$ be bounded. If $\sigma_{r-k}>\sigma_{r-k+1},$ then
		\begin{eqnarray}\label{per_SVD1}
		H_{r-k}(\tilde{\mathcal{A}})=H_{r-k}(\mathcal{A})+\displaystyle{O(\frac{1}{\lambda})},~\lambda\to\infty.
		\end{eqnarray}
\end{lemma}
\begin{proof}
	Apply (\ref{SVDA}) and (\ref{SVDA1}) with $B=\mathcal{A}$ and $\tilde{B}=\tilde{\mathcal{A}}$. Then by (iii) of Lemma~\ref{lemma 1}, we know that $E=\tilde{\mathcal{A}}-{\mathcal{A}}=O(\frac{1}{\lambda}).$ Indeed, with the non-increasing arrangement of the singular values in $\Sigma$ and $\tilde{\Sigma}$, and the fact that $E=\displaystyle{O(\frac{1}{\lambda})}$ as $\lambda\to\infty$, Lemma~\ref{lemma 3} immediately implies that
	\begin{align}\label{svd bound}
	\Sigma_1-\tilde{\Sigma}_1=\displaystyle{O(\frac{1}{\lambda})}~\;\;\text{and}\;\;\Sigma_2-\tilde{\Sigma}_2=O(\frac{1}{\lambda})\;\;\text{as}\;\lambda\to\infty.
	\end{align}
	Note that, ${\rm r}(\mathcal{A}_1)={\rm r}(\tilde{\mathcal{A}}_1)=r-k,$ and, since $\sigma_{r-k}>\sigma_{r-k+1},$ we can choose $\delta$ such that
	$$
	\delta=\frac{1}{2}(\sigma_{r-k}-\sigma_{r-k+1})>0.
	$$
In this way, for all large $\lambda$ the assumption of Lemma~\ref{lemma 4} will be satisfied. Since $\mathcal{A}_1=H_{r-k}(\mathcal{A})$ and  $\tilde{\mathcal{A}}_1=H_{r-k}(\tilde{\mathcal{A}}),$~(\ref{thresholding inequality}) can be written as
	\begin{eqnarray}\label{thresholding inequality1}
	\|H_{r-k}(\mathcal{A})-H_{r-k}(\tilde{\mathcal{A}})\|\le \|E\|(3+\frac{\|\mathcal{A}_2\|}{\delta}+\frac{\|\tilde{\mathcal{A}}_2\|}{\delta}).
	\end{eqnarray}
	Since ${\mathcal{A}_2}$ is fixed, $\|{\mathcal{A}_2}\|=O(1)$ as $\lambda\to\infty$. On the other hand, by~(\ref{svd bound}), as $\lambda\to\infty$, $$\tilde{\mathcal{A}_2}=\tilde{U}_2\tilde{\Sigma}_2\tilde{V}_2^T=\tilde{U}_2({\Sigma}_2+O(\frac{1}{\lambda}))\tilde{V}_2^T=\tilde{U}_2\Sigma_2\tilde{V}_2^T+O(\frac{1}{\lambda}\tilde{U}_2\tilde{V}_2^T).$$
	Now the unitary invariance of the matrix norm implies,
	\begin{align*} \|\tilde{\mathcal{A}_2}\|\le\|\tilde{U}_2\Sigma_2\tilde{V}_2^T\|+O(\frac{1}{\lambda}\|\tilde{U}_2\tilde{V}_2^T\|)=\|\Sigma_2\|+O(\frac{1}{\lambda}),
	\end{align*}
	which is bounded as $\lambda\to\infty$.
	Therefore~(\ref{thresholding inequality1}) becomes
	\begin{eqnarray}
	\|H_{r-k}(\mathcal{A})-H_{r-k}(\tilde{\mathcal{A}})\|\le C\|E\|,
	\end{eqnarray}
	for some constant $C>0$ and for all large $\lambda\to\infty$.
	Thus
	$$
	H_{r-k}(\tilde{\mathcal{A}})=H_{r-k}(\mathcal{A})+\displaystyle{O(\frac{1}{\lambda})},\lambda\to\infty,
	$$
	since $E=\displaystyle{O(\frac{1}{\lambda})}$ as $\lambda\to\infty$. This completes the proof of Lemma~\ref{lemma 2}.
\end{proof}

\noindent{\it Proof of Theorem~\ref{theorem 3}.} We first note that, for all $X_2$ with $r(\tilde{X}_1(W)\;~X_2)\leq r$,
$$
\|((A_1\;~A_2)-(\tilde{X}_1(W)\;~\tilde{X}_2(W)))\odot (W_1\;~\mathbbm{1})\|_F
\leq \|((A_1\;~A_2)-(\tilde{X}_1(W)\;~X_2))\odot (W_1\;~\mathbbm{1})\|_F.
$$
So, $\tilde{X}_2(W)$ solves
$$
\|A_2-\tilde{X}_2(W)\|_F
=\inf_{X_2:r(\tilde{X}_1(W)\;~X_2)\leq r} \|A_2-X_2\|_F.
$$
Thus, by Theorem~\ref{theorem 1}, $$
\tilde{X}_2(W)=P_{\tilde{X}_1(W)}(A_2)+H_{r-k}(P_{\tilde{X}_1(W)}^\perp(A_2)).$$
Therefore, using (ii) and (iii) of Lemma~\ref{lemma 1} and Lemma~\ref{lemma 2}, we get, as $\lambda\to\infty$,
$$
\tilde{X}_2(W)=P_{A_1}(A_2)+H_{r-k}(P_{A_1}^\perp(A_2))+O(\frac{1}{\lambda})=\tilde{A}_2+O(\frac{1}{\lambda}).$$
This, together with (i) of Lemma~\ref{lemma 1}, yields the desired result.~$\Box$\medskip

\noindent{\it Proof of Theorem~\ref{theorem 5}.}  Let $\hat{X}(W)=(\hat{X}_1(W)\;\;\hat{X}_2(W))$. We need to verify that $\{\hat{X}(W)\}_W$ is a bounded set and every accumulation point~(as $(W_1)_{ij}\to\infty$ and $(W_2)_{ij}\to 1$) is a solution to~(\ref{golub's problem}) for some $r$. Since $(\hat{X}_1(W)\;\;\hat{X}_2(W))$ is a solution to~(\ref{unconstrained hadamard problem}), we have
	\begin{align}\label{ineq2}
	&\|(A_1-\hat{X}_1(W))\odot W_1\|_F^2+ \|(A_2-\hat{X}_2(W))\odot W_2\|_F^2+\tau{\rm r}(\hat{X}_1(W)\;\;\hat{X}_2(W))\nonumber\\
	&\le\|(A_1-{X}_1)\odot W_1\|_F^2+\|(A_2-{X}_2)\odot W_2\|_F^2+\tau{\rm r}({X}_1\;\;{X}_2).
	\end{align}
	for all $({X}_1\;\;{X}_2).$
	%So, like in the proof of Theorem~\ref{theorem 4}, w
	By choosing $X_1=A_1,X_2=0$, we can obtain a constant $m_3:=\|A_2\odot W_2\|_F^2+\tau{\rm r}(A_1\;\;0)$ such that $\|(A_1-\hat{X}_1(W))\odot W_1\|_F^2+\|(A_2-\hat{X}_2(W))\odot W_2\|_F^2\le m_3.$ Therefore, $\{\hat{X}_1(W)\;\;\hat{X}_2(W)\}$ is bounded. Let $(X_1^{**}\;\;X_2^{**})$ be an accumulation point of the sequence for some subsequences as $(W_1)_{ij}\to\infty$ and $(W_2)_{ij}\to 1$. We only need to show that $(X_1^{**}\;\;X_2^{**})\in\displaystyle{\mathop{\cup}_{r}{\cal A}_G^r}.$ As in the proof of Lemma~\ref{lemma 1}~(i), we can show that
	\begin{align}\label{2}
	\lim_{\substack{(W_1)_{ij}\to\infty\\(W_2)_{ij}\to 1}}\hat{X}_1(W)=A_1.
	\end{align}
Now, taking limit and setting $X_1=A_1$ in~(\ref{ineq2}), we obtain,
	\begin{align}\label{limitcase}
	&\|A_2-X_2^{**}\|_F^2+\tau{\rm r}(A_1\;\;{X}_2^{**})\le\|A_2-{X}_2\|_F^2+\tau{\rm r}(A_1\;\;{X}_2),
	\end{align}
	for all ${X}_2.$ If we denote $r^{**}={\rm r}(A_1\;\;{X}_2^{**}),$ then for ${X}_2$ with $ {\rm r}(A_1\;\;{X}_2)\le r^{**},$~(\ref{limitcase}) yields
	\begin{align}
	\|A_2-X_2^{**}\|_F^2\le\|A_2-{X}_2\|_F^2.
	\end{align}
	Therefore, $X_2^{**}$ is a solution to the problem of Golub, Hoffman, and Stewart and by Theorem~\ref{theorem 1},
	$$
	X_2^{**}=P_{A_1}(A_2)+H_{r^{**}-k}\left(P^{\perp}_{A_1}(A_2)\right).
	$$
	This, together with~(\ref{2}) completes the proof.~$\Box$
	
\noindent{\it Proof of Theorem~\ref{theorem 6}.}  For convenience, we will drop the dependence on $W$ in our notations.
	Let $\hat{X}=(\hat{X}_1\;\;\hat{X}_2)$ solve the minimization problem~(\ref{unconstrained hadamard problem}).~As in the proof of Lemma~\ref{lemma 1} (i) again we have
	\begin{eqnarray}\label{R_1}
	\hat{X}_1\to A_1,~{\rm as}~(W_1)_{ij}\to\infty.
	\end{eqnarray}
Next, we show the convergence of $\hat{X}_2$: for $r^*$ satisfying $\sigma_{r^*}^2>\tau\geq \sigma_{r^*+1}^2$,
\begin{equation}\label{x2}
\hat{X}_2\to P_{A_1}(A_2)+H_{r^*}(P_{A_1}^\perp(A_2)), ~{\rm as}~(W_1)_{i,j}\to\infty,~(W_2)_{i,j}\to 1.
\end{equation}
 By Theorem~\ref{theorem 5}, $\hat{X}$ is bounded. So, to establish (\ref{x2}), we will need only to show (i) there is only one accumulation point as $(W_1)_{i,j}\to\infty,~(W_2)_{i,j}\to 1$, and (ii) this unique accumulation point is $P_{A_1}(A_2)+H_{r^*}(P_{A_1}^\perp(A_2))$ for $r^*$ satisfying $\sigma_{r^*}^2>\tau\geq \sigma_{r^*+1}^2$.

Let $X_2^*$ be any accumulation point of $\hat{X}_2$ as  $(W_1)_{i,j}\to\infty,~(W_2)_{i,j}\to 1$. 

Choosing $X_1=\hat{X}_1$ in~(\ref{ineq2})~we find, for all $X_2$,
\begin{eqnarray}\label{weighted unconstrained problem 5}
	\;\;\;\;\;\;\|(A_2-\hat{X}_2)\odot W_2\|_F^2+\tau{\rm r}(
	\hat{X}_{1}\;\;\hat{X}_{2})
	\le \|(A_2-{X}_2)\odot W_2\|_F^2+\tau{\rm r}(
	\hat{X}_{1}\;\;{X}_{2}).
	\end{eqnarray}
We will apply some ideas from Golub, Hoffman, and Stewart ~\cite{golub}. Since the weight $W_2$ gets in
the way (by destroying the unitary invariance of the norm), we have to take the limit $(W_2)_{i,j}\to 1$ at the right time.

As in~\cite{golub}, assume ${\rm r}(A_1)=k$ and consider a $QR$ decomposition of $A$ and corresponding decomposition of $\hat{X}$ and $X$:
	$$
	A=QR=(Q_1~~Q_2)\begin{pmatrix}{R}_{11}&{R}_{12}\\0& R_{22}\\0& 0\end{pmatrix}=(Q_1R_{11}~~Q_1R_{12}+Q_2\begin{pmatrix}R_{22}\\ 0\end{pmatrix})=(A_1~~A_2),
	$$
 let $$\hat{R}:=Q^T\hat{X}=\begin{pmatrix}\hat{R}_{11}&\hat{R}_{12}\\\hat{R}_{21}&\hat{R}_{22}\\\hat{R}_{31}&\hat{R}_{32}\end{pmatrix}$$ and let  $${R}^{\dagger}:=Q^TX=\begin{pmatrix}{R}^{\dagger}_{11}&{R}^{\dagger}_{12}\\{R}^{\dagger}_{21}&{R}^{\dagger}_{22}\\{R}^{\dagger}_{31}&{R}^{\dagger}_{32}\end{pmatrix}.$$ 
 Since the rank of a matrix is invariant under an unitary transformation,~(\ref{weighted unconstrained problem 5}) can be rewritten as
	\begin{eqnarray}\label{weighted unconstrained problem 6}
	&\|(A_2-\hat{X}_2)\odot W_2\|_F^2+\tau{\rm r}(
	Q^T\hat{X}_{1}\;\;Q^T\hat{X}_{2})\nonumber\\
	&\le \|(A_2-{X}_2)\odot W_2\|_F^2+\tau{\rm r}(
	Q^T\hat{X}_{1}\;\;Q^T{X}_{2}).
	\end{eqnarray}
When $\lambda=\min_{i,j}(W_1)_{i,j}$ is large enough, $\hat{R}_{11}$ is nonsingular by~(\ref{R_1}).~Using ${\rm r}(A_1)=k$~%(if $A_1$ is not of full rank we can not use the invertibility of the matrix $\hat{R}_{11}$),
and performing the row and column operations on the second terms on both sides of~(\ref{weighted unconstrained problem 6}), we get
\begin{align}\label{weighted unconstrained problem 7}
&\|(A_2-\hat{X}_2)\odot W_2\|_F^2+\tau {\rm r}\begin{pmatrix}
\hat{R}_{22}-\hat{R}_{21}\hat{R}_{11}^{-1}\hat{R}_{12} \\
\hat{R}_{32}-\hat{R}_{31}\hat{R}_{11}^{-1}\hat{R}_{12}\\
\end{pmatrix}\nonumber\\&\le\|(A_2-{X}_2)\odot W_2\|_F^2+\tau {\rm r}\begin{pmatrix}
{R}^{\dagger}_{22}-\hat{R}_{21}\hat{R}_{11}^{-1}{R}^{\dagger}_{12} \\
{R}^{\dagger}_{32}-\hat{R}_{31}\hat{R}_{11}^{-1}{R}^{\dagger}_{12}\\
\end{pmatrix},
\end{align}
for all $R^{\dagger}_{12},{R}^{\dagger}_{22},~{\rm and}~{R}^{\dagger}_{32}.$

 Note that $\displaystyle{\begin{pmatrix}\hat{R}_{12}^{*}\\\hat{R}_{22}^{*}\\\hat{R}_{32}^{*}\end{pmatrix}}:=Q^T\hat{X}^*_2$ is an accumulation point for a subsequence, $\Lambda$ say, of $\begin{pmatrix}\hat{R}_{12}\\\hat{R}_{22}\\\hat{R}_{32}\end{pmatrix}=Q^T\hat{X}_2.$ From~(\ref{R_1}), using the fact that $\hat{R}_{11}\to R_{11},\hat{R}_{21}\to 0,~{\rm and}~\hat{R}_{31}\to 0,$ as $(W_1)_{ij}\to\infty, (W_2)_{ij}\to1$ in~(\ref{weighted unconstrained problem 7}) we get, by taking limit along the subsequence $\Lambda$,
\begin{align}\label{weighted unconstrained problem 8}
&\|A_2-\hat{X}_2^*\|_F^2+\tau {\rm r}\begin{pmatrix}
\hat{R}_{22}^*\\
\hat{R}_{32}^*\\
\end{pmatrix}
\le\|A_2-{X}_2\|_F^2+\tau {\rm r}\begin{pmatrix}
{R}^{\dagger}_{22}\\
{R}^{\dagger}_{32}\\
\end{pmatrix},
\end{align}
for all ${R}^{\dagger}_{12},{R}^{\dagger}_{22},~{\rm and}~{R}^{\dagger}_{32}.$ Since Frobenius norm is unitarily invariant, (\ref{weighted unconstrained problem 8}) reduces to
\begin{align}\label{weighted unconstrained problem 9}
&\|\begin{pmatrix} R_{12}\\
R_{22}\\
0\\
\end{pmatrix}-\begin{pmatrix} \hat{R}_{12}^*\\
\hat{R}_{22}^*\\
\hat{R}_{32}^*\\
\end{pmatrix}\|_F^2+\tau {\rm r}\begin{pmatrix}
\hat{R}_{22}^*\\
\hat{R}_{32}^*\\
\end{pmatrix}
\le\|\begin{pmatrix} R_{12}\\
R_{22}\\
0\\
\end{pmatrix}-\begin{pmatrix} {R}_{12}^{\dagger}\\
{R}_{22}^{\dagger}\\
{R}_{32}^{\dagger}\\
\end{pmatrix}\|_F^2+\tau {\rm r}\begin{pmatrix}
{R}^{\dagger}_{22}\\
{R}^{\dagger}_{32}\\
\end{pmatrix},
\end{align}
for all ${R}^{\dagger}_{12},{R}^{\dagger}_{22},~{\rm and}~{R}^{\dagger}_{32}.$
Substituting ${R}^{\dagger}_{22}=\hat{R}_{22}^*, {R}^{\dagger}_{32}=\hat{R}_{32}^*$, and~$R^{\dagger}_{12}=R_{12}$, in~(\ref{weighted unconstrained problem 9}) yields
\begin{align*}
&\|{R}_{12}-\hat{R}^*_{12}\|_F^2
\le 0,
\end{align*}
and so $R^*_{12}=R_{12}$. Now, substituting $R^{\dagger}_{12}=\hat{R}^*_{12}$ in~(\ref{weighted unconstrained problem 9}) we find
 \begin{align}\label{weighted unconstrained problem 11}
&\|\begin{pmatrix}{R}_{22}\\0\end{pmatrix}-\begin{pmatrix}
\hat{R}_{22}^*\\
\hat{R}_{32}^*\\
\end{pmatrix}\|_F^2+\tau {\rm r}\begin{pmatrix}
\hat{R}_{22}^*\\
\hat{R}_{32}^*\\
\end{pmatrix}
\le\|\begin{pmatrix}{R}_{22}\\0\end{pmatrix}-\begin{pmatrix}
{R}^{\dagger}_{22}\\
{R}^{\dagger}_{32}\\
\end{pmatrix}\|_F^2+\tau {\rm r}\begin{pmatrix}
{R}^{\dagger}_{22}\\
{R}^{\dagger}_{32}\\
\end{pmatrix},
\end{align}
for all ${R}^{\dagger}_{22},{R}^{\dagger}_{32}$. Let $\bar{R}^*=\begin{pmatrix}
\hat{R}_{22}^{*}\\\hat{R}_{32}^{*}\end{pmatrix}$ and $r^*={\rm r}(\bar{R}^*),$ then~(\ref{weighted unconstrained problem 11}) implies
\begin{align}\label{R22lowrank}
\|\begin{pmatrix}{R}_{22}\\0\end{pmatrix}-\bar{R}^*\|_F^2\le \|\begin{pmatrix}{R}_{22}\\0\end{pmatrix}-{R}^*\|_F^2,
\end{align}
for all ${R}^*\in\mathbb{R}^{(m-k)\times(n-k)}$ with ${\rm r}({R}^*)\le r^*.$~So $\bar{R}^*$ solves a problem of classical low rank approximation of $\begin{pmatrix}{R}_{22}\\0\end{pmatrix}$. Note that, $Q_2\begin{pmatrix}{R}_{22}\\0\end{pmatrix}=P^{\perp}_{A_1}(A_2)$~(see~Theorem 1 in~\cite{golub}). Since $P^{\perp}_{A_1}(A_2)$ has distinct singular values,~there exists a unique $\bar{R}^*$ which is given by $\bar{R}^*=H_{r^*}\begin{pmatrix}R_{22}\\0\end{pmatrix}$ as in~(\ref{hardthresholding}).~Therefore,
 $$
\lim_{\Lambda} Q^T\hat{X}_2 = Q^T\hat{X}_2^*=\begin{pmatrix}R_{12}\\ H_{r^*}\begin{pmatrix}R_{22}\\0\end{pmatrix}\end{pmatrix},$$
 which, together with~(\ref{R_1}), 
 implies
\begin{align*}
%\displaystyle{\lim_{\substack{(W_1)_{ij}\to\infty\\(W_2)_{ij}\to1}}(\hat{X}_1\;\;\hat{X}_2)}
\displaystyle{\lim_{\Lambda}}(\hat{X}_1\;\;\hat{X}_2)
&=Q\begin{pmatrix}  & R_{12}\\
R_1  & H_{r^*}\begin{pmatrix}
R_{22}\\
0
\end{pmatrix}
\end{pmatrix}\\
&=(A_1\;\;\;Q_1R_{12}+H_{r^*}\left(Q_2\begin{pmatrix}
R_{22}\\
0
\end{pmatrix}\right)),
\end{align*}
which is the same as
$$
\left(A_1\;\;\;\;P_{A_1}(A_2)+H_{r^*}\left(P^{\perp}_{A_1}(A_2)\right)\right).
$$

Finally, if we can show that $r^*$ does not depend on $\Lambda$ then we know that all accumulation points equal and therefore, we can complete the proof. In fact, $r^*$ depends only on $\tau$ as we see from the argument below.

Assume that
$$
\begin{pmatrix}{R}_{22}\\0\end{pmatrix}=U\Sigma V^T
$$ is a SVD of $\begin{pmatrix}{R}_{22}\\0\end{pmatrix}.$ Then, for any ${R}^*\in\mathbb{R}^{(m-k)\times (n-k)},$~(\ref{weighted unconstrained problem 11}) gives
\begin{eqnarray}\label{QR inequality}
&\|\Sigma-U^T\bar{R}^*V\|_F^2+\tau {\rm r}(U^T\bar{R}^*V)\nonumber\\
&\le\|\Sigma-U^T{R}^*V\|_F^2+\tau{\rm r}(U^T{R}^*V),
\end{eqnarray}
Since $r^*={\rm r}(\bar{R}^*)$ and $
U^T\bar{R}^{*} V={\rm diag}(\sigma_1\;\sigma_2\;\cdots\sigma_{r^*}\;0\cdots 0)$,
choosing ${R}^*$ such that
$$
U^T{R}^*V={\rm diag}(\sigma_1\;\sigma_2\;\cdots\sigma_{r^*+1}\;0\cdots 0),
$$
and using~(\ref{QR inequality}) we find
\begin{align*}
\sigma_{r^*+2}^2+\cdots+\sigma_n^2+\tau\ge\sigma_{r^*+1}^2+\sigma_{r^*+2}^2+\cdots+\sigma_n^2.
\end{align*}
Next we choose ${R}^*$ such that
$$
U^T{R}^*V={\rm diag}(\sigma_1\;\sigma_2\;\cdots\sigma_{r^*-1}\;0\cdots 0),
$$
and so ${\rm r}(R^*)=r^*-1<r^*$. Now~(\ref{R22lowrank}) and Eckart-Young-Mirsky's theorem imply the equality in~(\ref{QR inequality}) can not hold. So,
\begin{align*}
\sigma_{r^*}^2+\cdots+\sigma_n^2-\tau>\sigma_{r^*+1}^2+\sigma_{r^*+2}^2+\cdots+\sigma_n^2.
\end{align*}
Therefore, we obtain
\begin{eqnarray}\label{sigma inequality}
\sigma_{r^*}^2>\tau\ge\sigma_{r^*+1}^2.
\end{eqnarray}
It is easy to see that if (\ref{sigma inequality}) holds then ${\rm r}(\bar{R}^*)=r^*.$
So,
$$
{\rm r}(\bar{R}^*)=r^*\;\text{if and only if}\;\sigma_{r^*}^2>\tau\ge\sigma_{r^*+1}^2.
$$
This completes the proof.~$\Box$
\section{Numerical algorithm}

In this section we propose a numerical algorithm to solve (\ref{hadamard problem}) when $W_2=\mathbbm{1}$, which, in general, does not have a closed form solution~\cite{srebro,manton}.
Note~(\ref{hadamard problem}) can be written as~(with $W_2=\mathbbm{1}$)
\begin{align*}
\min_{\substack{X_1,X_2\\{\rm r}(X_1\;\;X_2)\le r}}\left(\|(A_1-X_1)\odot W_1\|_F^2+\|A_2-X_2\|_F^2\right).
\end{align*}
We assume that ${\rm r}(X_1)=k$. It can be verified that any $X_2$ such that ${\rm r}(X_1\;\;X_2)\le r$ can be given in the form
$$X_2=X_1C+BD,$$ for some arbitrary matrices $B\in\mathbb{R}^{m\times (r-k)},$ $D\in\mathbb{R}^{(r-k)\times (n-k)},$ and $C\in\mathbb{R}^{k\times (n-k)}.$ Hence we need to solve
\begin{align}\label{main problem 2}
\min_{X_1,C,B,D}\left(\|(A_1-X_1)\odot W_1\|_F^2+\|A_2-X_1C-BD\|_F^2\right).
\end{align}
Note that, using a block structure, we can write (\ref{main problem 2}) as a weighted low rank approximation (with a special low rank structure):
\begin{align*}
\min_{X_1,C,B,D}\left\{\|\left((A_1\;\;A_2)-(X_1\;\;B)\begin{pmatrix} I_k & C\\0 & D\end{pmatrix}\right)\odot (W_1\;\;\mathbbm{1})\|_F^2\right\},
\end{align*}
which is a form of the alternating weighted least squares problem in the literature~\cite{srebro, markovosky}. But we will not follow the general algorithm proposed in~\cite{markovosky}~for the following reasons:~(i)~due to the special structure of the weight, our algorithm is more efficient than~\cite{markovosky}~(see Algorithm 3.1,~in p.~42~\cite{markovosky}),~(ii) it allows a detailed convergence analysis which is usually not available in other algorithms proposed in the literature~\cite{srebro,manton,markovosky}, and (iii) it can handle bigger size matrices as we will demonstrate in the numerical result section.
% One is to verify the rate given by Theorem~\ref{theorem 3} numerically and to gain some insight on the sharpness of the rate~($O(\frac{1}{\lambda})$, as $\lambda\to\infty$); the other one is to demonstrate a fast and simple numerical procedure based on alternating direction method in solving the weighted low-rank approximation problem that also
If $k=0,$ then~(\ref{main problem 2}) reduces to an unweighted rank $r$ factorization of $A_2$ and can be solved as an alternating least squares problem~\cite{mazumder,byod,hansohm}.

\subsection{Optimization procedure}
Denote $F(X_1,C,B, D)=\|(A_1-X_1)\odot W_1\|_F^2+\|A_2-X_1C-BD\|_F^2$ as the objective function. The above problem can be numerically solved by using an alternating strategy~\cite{LinChenMa, Liu} of minimizing the function with respect to each component iteratively:
\begin{align}\label{update rule}
\left\{\begin{array}{ll}
\displaystyle{(X_1)_{p+1}=\arg\min_{X_1}F(X_1,C_p,B_p,D_p)},\\
\displaystyle{C_{p+1}=\arg\min_{C}F((X_1)_{p+1},C,B_p,D_p)},\\
\displaystyle{B_{p+1}=\arg\min_{B}F((X_1)_{p+1},C_{p+1},B,D_p)},\\
\text{and,}\; \displaystyle{D_{p+1}=\arg\min_{D}F((X_1)_{p+1},C_{p+1},B_{p+1},D)}.
\end{array}\right.
\end{align}
Note that each of the minimizing problem for $X_1, C, B,$ and $D$ can be solved explicitly by looking at the gradients of $F(X_1,C,B,D)$. But finding an update rule for $X_1$ turns out to be more involved than the other three variables due to the interference of the weight $W_1$. We update $X_1$ element wise along each row. Therefore we will use the notation $X_1(i,:)$ to denote the $i$-th row of the matrix $X_1$. We set $\frac{\partial}{\partial X_1}F(X_1,C_p,B_p,D_p)|_{X_1=(X_1)_{p+1}}=0$  and obtain
\begin{align}\label{x_p}
-(A_1-(X_1)_{p+1})\odot W_1\odot W_1-(A_2-(X_1)_{p+1}C_p-B_pD_p)C_p^T=0.
\end{align}
Solving the above expression for $X_1$ sequentially along each row gives
\begin{align*}
(X_1(i,:))_{p+1}=(E(i,:))_p({\rm diag}(W_1^2(i,1)\;W_1^2(i,2)\cdots W_1^2(i,k))+C_pC_p^T)^{-1},
\end{align*}
where $E_p=A_1\odot W_1\odot W_1+(A_2-B_pD_p)C_p^T$. Note that, for each row $X_1(i,:)$, if $\mathcal{L}_i={\rm diag}(W_1^2(i,1)\;W_1^2(i,2)\cdots W_1^2(i,k))+C_pC_p^T$ then the above system of equations are equivalent to solving a least squares solution of $\mathcal{L}_i(X_1(i,:))_{p+1}^T=(E(i,:))_p^T$ for each $i$.
Next we find $C_{p+1}$ by setting
$$
\frac{\partial}{\partial C}F(X_1,C,B_p,D_p)|_{C = C_{p+1}}=0,
$$
which implies
\begin{align}\label{C_p}
-(X_1)_{p+1}^T(A_2-(X_1)_{p+1}C_{p+1}-B_pD_p)=0,
\end{align}
and consequently solving for $C_{p+1}$ gives~(assuming $(X_1)_{p+1}$ is of full rank)
$$
C_{p+1}=((X_1)_{p+1}^T(X_1)_{p+1})^{-1}((X_1)_{p+1}^TA_2-(X_1)_{p+1}^TB_pD_p).
$$
Similarly, $B_{p+1}$ satisfies
\begin{align}\label{B_p}
-(A_2-(X_1)_{p+1}C_{p+1}-B_{p+1}D_p)D_p^T=0.
\end{align}
Solving~(\ref{B_p}) for $B_{p+1}$ obtains~(assuming $D_p$ is of full rank)
$$
B_{p+1}=(A_2D_p^T-(X_1)_{p+1}C_{p+1}D_p^T)(D_pD_p^T)^{-1}.
$$
Finally, $D_{p+1}$ satisfies
\begin{align}\label{D_p}
-B_{p+1}^T(A_2-(X_1)_{p+1}C_{p+1} -B_{p+1}D_{p+1})=0,
\end{align}
and we can write~(assuming $B_{p+1}$ is of full rank)
$$
D_{p+1}=(B_{p+1}^TB_{p+1})^{-1}(B_{p+1}^TA_2-B_{p+1}^T(X_1)_{p+1}C_{p+1}).
$$
We arrive at the following algorithm.
\begin{algorithm}
	\SetAlgoLined
	\SetKwInOut{Input}{Input}
	\SetKwInOut{Output}{Output}
	\SetKwInOut{Init}{Initialize}
	\nl\Input{$A=(A_1\;\;A_2) \in\mathbb{R}^{m\times n}$ (the given matrix); $W= (W_1\;\;\mathbbm{1})\in\mathbb{R}^{m\times n}$ (the weight), threshold $\epsilon>0$\;}
	\nl\Init {$(X_1)_0,C_0,B_0,D_0$\;}
	\BlankLine
	\nl \While{not converged}
	{
		\nl $E_p=A_1\odot W_1\odot W_1+(A_2-B_pD_p)C_p^T$\;
		%\BlankLine
		\nl \For {$i=1:m$}
		{	
			\nl $\displaystyle{(X_1(i,:))_{p+1}=(E(i,:))_p({\rm diag}(W_1^2(i,1)\;W_1^2(i,2)\cdots W_1^2(i,k))+C_pC_p^T)^{-1}}$\;
		}
		%\BlankLine
		\nl $C_{p+1}=((X_1)_{p+1}^T(X_1)_{p+1})^{-1}(X_1)_{p+1}^T(A_2-B_pD_p)$\;
		\nl $B_{p+1}=(A_2-(X_1)_{p+1}C_{p+1})D_p^T(D_pD_p^T)^{-1}$\;
		\nl $D_{p+1}=(B_{p+1}^TB_{p+1})^{-1}B_{p+1}^T(A_2-(X_1)_{p+1}C_{p+1})$\;
		\nl $p=p+1$\;
	}
	\BlankLine
	\nl \Output{$(X_1)_{p+1}, (X_1)_{p+1}C_{p+1}+B_{p+1}D_{p+1}.$}
	\caption{WLR Algorithm}
\end{algorithm}
\subsection{Complexity of the algorithm}
In this section we discuss the runtime complexity of Algorithm 1 by making some simplifying assumptions.~The update in Step 4 takes $O(mk(n-k))+O(m(r-k)(n-k))$ time. The matrix inversion in Step 6 for updating each row of $X_1$ takes $O(k^3)$ time and the total computational time for $m$ rows is $O(mk^3)+O(mk(2k-1))$. Indeed, the costly procedures in Step 7, 8,~and 9 are the matrix inversions. The matrix product and inversion in Step 7 takes $O(k^2(m+k))$ time, where the inner matrix product takes $O(m(r-k)(n-k)+mk(n-k))$ and the final product takes $O(k^2(n-k))$ time. In Step 8, the inner matrix product takes $O(k(n-k)(r-k)+m(r-k)(n-k))$ time. The matrix product and inversion takes $O((r-k)^2(r+n-2k))$ time and the final product takes $O(m(r-k)^2)$ time.
Finally in Step 9, as before the main cost is due to the matrix product and inversion, which takes $O((r-k)^2(m+r-k))$ time. The matrix product in Step 9 takes $O((r-k)m(n-k))+O((r-k)^2(n-k))$ time. In summary, the total complexity of the algorithm is $O(mk^3+mnr)$.

\subsection{Convergence analysis} Next we will discuss the convergence of our numerical algorithm. Since the objective function $F$ is convex only in each of the component $X_1,B,C,$ and $D$, it is hard to argue about the global convergence of the algorithm. In Theorem~\ref{theorem 8} and \ref{theorem 9}, under some special assumptions when the limit of the individual sequence exists, we show that the limit points are going to be a stationary point of $F$. To establish our main convergence results in Theorem~\ref{theorem 8} and \ref{theorem 9}, the following equality will be very helpful.

\begin{theorem}\label{theorem 7}
	For a fixed $(W_1)_{ij}>0$ let $m_p=F((X_1)_p,C_p,B_p,D_p)$ for $p=1,2,\cdots$
.
	Then,\begin{align}\label{mj equality}
	m_p-m_{p+1}=&\|((X_1)_p-(X_1)_{p+1})\odot W_1\|^2_F+\|((X_1)_p-(X_1)_{p+1})C_p\|^2_F\nonumber\\&+\|(X_1)_{p+1}(C_p-C_{p+1})\|^2_F
	+\|(B_p-B_{p+1})D_p\|^2_F+\|B_{p+1}(D_p-D_{p+1})\|^2_F.
	\end{align}
\end{theorem}
\begin{proof}
	Denote
	\begin{align}\label{update rule 1}
	\left\{\begin{array}{ll}
	m_p-F((X_1)_{p+1},C_{p},B_{p},D_{p})=d_1,\\
	F((X_1)_{p+1},C_{p},B_{p},D_{p})-F((X_1)_{p+1},C_{p+1},B_{p},D_{p})=d_2,\\
	F((X_1)_{p+1},C_{p+1},B_{p},D_{p})-F((X_1)_{p+1},C_{p+1},B_{p+1},D_{p})=d_3,\\
	\text{and,}\; F((X_1)_{p+1},C_{p+1},B_{p+1},D_{p})-m_{p+1}=d_4.
	\end{array}\right.
	\end{align}
	Therefore,
	\begin{align}\label{update rule calc}
	d_1&=\|(A_1-(X_1)_{p})\odot W_1\|_F^2+\|A_2-(X_1)_pC_p-B_pD_p\|_F^2-\|(A_1-(X_1)_{p+1})\odot W_1\|_F^2\nonumber\\
	&-\|A_2-(X_1)_{p+1}C_p-B_pD_p\|_F^2\nonumber\\
	&=\|(X_1)_{p}\odot W_1\|_F^2-\|(X_1)_{p+1}\odot W_1\|_F^2+\|(X_1)_{p}C_p\|_F^2-\|(X_1)_{p}C_{p+1}\|_F^2
	+2\langle A_1\odot W_1\odot W_1,\nonumber\\&(X_1)_{p+1}-(X_1)_{p}\rangle-2\langle ((X_1)_{p}-(X_1)_{p+1})C_p,A_2-B_pD_p\rangle.
	\end{align}
	Note that,
	\begin{align*}
	(((X_1)_{p+1}-A_1)\odot W_1\odot W_1)((X_1)_{p}-(X_1)_{p+1})^T\\=(A_2-(X_1)_{p+1}C_p-B_pD_p)C_p^T((X_1)_{p}-(X_1)_{p+1})^T,
	\end{align*}
	as $(X_1)_{p+1}$ satisfies~(\ref{x_p}).
	This, together with~(\ref{update rule calc}), will lead us to
	\begin{align}\label{update rule calc 1}
	d_1=\|((X_1)_p-(X_1)_{p+1})\odot W_1\|^2_F+\|((X_1)_p-(X_1)_{p+1})C_p\|^2_F.
	\end{align}
	Similarly we find
	\begin{align}\label{D}
	\left\{\begin{array}{ll}
	d_2=\|(X_1)_{p+1}(C_p-C_{p+1})\|^2_F,\\
	d_3=\|(B_p-B_{p+1})D_p\|^2_F,\\
	d_4=\|B_{p+1}(D_p-D_{p+1})\|^2_F.
	\end{array}\right.
	\end{align}
	Combining them together we have the desired result.
	\qquad\end{proof}

\begin{remark}\label{m_p}
From Theorem~\ref{theorem 7} we know that the non-negative sequence $\{m_p\}$ is non-increasing.~Therefore, $\{m_p\}$ has a limit. 
\end{remark}

Theorem~\ref{theorem 7} implies a lot of interesting convergence properties of the algorithm.~For example, we have the following estimates.
\noindent\begin{corollary}\label{corollary 2}
	We have \begin{enumerate}[(i)]
		\item $m_p- m_{p+1}\geq \frac{1}{2}\|B_{p+1}D_{p+1}-B_pD_p\|_F^2$ for all $p.$
		\item $m_p- m_{p+1}\geq\|((X_1)_p-(X_1)_{p+1})\odot W_1\|^2_F$ for all $p$.
	\end{enumerate}
\end{corollary}
\noindent \begin{proof}{\it (i).}
	From~(\ref{mj equality}) we can write, for all $p$,
	\begin{align*}
	& m_p-m_{p+1}\\
	&\geq \|(B_p-B_{p+1})D_p\|^2_F+\|B_{p+1}(D_p-D_{p+1})\|^2_F\\
	&=\frac{1}{2}(\|B_{p+1}D_{p+1}-B_pD_p\|^2_F+\|2B_{p+1}D_p-B_{p+1}D_{p+1}-B_pD_p\|^2_F),
	\end{align*}
	by parallelogram identity. Therefore,
	\begin{align*}
	& m_p-m_{p+1}\geq\frac{1}{2}\|B_{p+1}D_{p+1}-B_pD_p\|_F^2.
	\end{align*}
	This completes the proof of (i).\\
	{\it (ii).} This follows immediately from~(\ref{mj equality}).
	\qquad\end{proof}

We now can state our main convergence results as a consequence of Theorem~\ref{theorem 7} and Corollary~\ref{corollary 2}.
\begin{theorem}\label{theorem 8}
	\begin{enumerate}[(i)]
		\item  We have the following:~$\sum_{p=1}^{\infty}\|B_{p+1}D_{p+1}-B_pD_p\|_F^2<\infty$, and\\ $\displaystyle{\sum_{p=1}^{\infty}\left(\|((X_1)_p-(X_1)_{p+1})\odot W_1\|_F^2\right)<\infty}$.
		\item If $\sum_{p=1}^{\infty}\sqrt{m_p- m_{p+1}}< +\infty,$ then $\displaystyle{\lim_{p\to\infty}B_pD_p}$ and $\displaystyle{\lim_{p\to\infty}(X_1)_p}$ exist. Furthermore if we write $L^*:=\displaystyle{\lim_{p\to\infty}B_pD_p}$ then $\displaystyle{\lim_{p\to\infty}B_{p+1}D_{p}=L^*}$ for all $p.$
	\end{enumerate}
\end{theorem}

\noindent\begin{proof}{\it (i).} From Corollary~\ref{corollary 2}~(i)
	we can write, for $N>0$,
	$$
	2(m_1-m_{N+1})\ge\sum_{p=1}^{N}\left(\|B_{p+1}D_{p+1}-B_pD_p\|_F^2\right),$$
and from Corollary~\ref{corollary 2}~(ii) and using $\lambda=\displaystyle{\min_{\substack{1\le i\le m\\1\le j\le k}}(W_1)_{ij}}$ we have
$$m_1-m_{N+1}\ge\displaystyle{\sum_{p=1}^{N}\left(\|((X_1)_p-(X_1)_{p+1})\odot W_1\|_F^2\right)}
	\ge\displaystyle{\lambda^2\sum_{p=1}^{N}\|(X_1)_p-(X_1)_{p+1}\|_F^2}.$$
Now note that, by Remark~\ref{m_p}, $\{m_p\}_{p=1}^{\infty}$ is a convergent sequence. Hence the result follows.\\
	\noindent {\it (ii).} Again using Corollary~\ref{corollary 2} (i) we can write, for $N>0$,
	\begin{align*}
	&\sum_{p=1}^{N}\sqrt{m_p- m_{p+1}}
	\ge \frac{1}{\sqrt{2}}\sum_{p=1}^{N}\left(\|B_{p+1}D_{p+1}-B_pD_p\|_F\right),
	\end{align*}
	which implies $\sum_{p=1}^{\infty}(B_{p+1}D_{p+1}-B_pD_p)$ is convergent if $\sum_{p=1}^{\infty}\sqrt{m_p- m_{p+1}}< +\infty.$ Therefore, $\displaystyle{\lim_{N\to\infty}B_ND_N}$ exists. Similarly, we can conclude that $\displaystyle{\lim_{p\to\infty}(X_1)_p}$ exists.\\
	Further, $\displaystyle{\lim_{p\to\infty}\|B_{p+1}D_{p+1}}-B_{p+1}D_p\|_F^2=0,$ as implied by~(i) above. 
	Therefore $\displaystyle{\lim_{p\to\infty}B_{p+1}D_p}$ exists and is equal to $\displaystyle{\lim_{p\to\infty}B_pD_p}=L^*.$ 
	This completes the proof. 
	\qquad\end{proof}

From Corollary~\ref{theorem 8}, we can only prove the convergence of the sequence $\{B_pD_p\}$ but not of $\{B_p\}$ and $\{D_p\}$ separately. We next establish the convergence of $\{B_p\}$ and $\{D_p\}$ with stronger assumption. Consider the situation when
\begin{align}\label{mj inequality}
\sum_{p=1}^{\infty}\sqrt{m_p- m_{p+1}}< +\infty.
\end{align}
\begin{theorem} \label{theorem 9}
	Assume~(\ref{mj inequality}) holds.
	\begin{enumerate}[(i)]
		\item If $B_p$ is of full rank and $B_p^TB_p\ge\gamma I_{r-k}$ for large $p$ and some $\gamma>0$ then $\displaystyle{\lim_{p\to\infty}D_p}$ exists.
		\item If $D_p$ is of full rank and $D_pD_p^T\ge\delta I_{r-k}$ for large $p$ and some $\delta>0$  then $\displaystyle{\lim_{p\to\infty}B_p}$ exists.
		\item If $X_1^*:=\displaystyle{\lim_{p\to\infty}(X_1)_{p}}$ is of full rank, then $C^*:=\displaystyle{\lim_{p\to\infty}C_p}$ exists. Furthermore, if we write $L^*=B^*D^*,$ for $B^*\in\mathbb{R}^{m\times(r-k)}, D^*\in\mathbb{R}^{(r-k)\times(n-k)}$, then $({X}_1^*,C^*,B^*,D^*)$ will be a stationary point of $F$.
	\end{enumerate}
\end{theorem}
\noindent \begin{proof}{\it (i).} Using~(\ref{mj equality}) we have, for $N>0$,
	\begin{align*}
	\sum_{p=1}^{N}\sqrt{m_p-m_{p+1}}&\geq \sum_{p=1}^{N}\|B_{p+1}(D_p-D_{p+1})\|_F\\
	&=\sum_{p=1}^{N }\sqrt{{\rm tr}[(D_p-D_{p+1})^TB_{p+1}^TB_{p+1}(D_p-D_{p+1})]},
	\end{align*}
	where ${\rm tr}(X)$ denotes the trace of the matrix $X$.
	Note that $B_p^TB_p\ge\gamma I_{r-k},$ and so we obtain
	\begin{align*}
	\sum_{p=1}^{N}\sqrt{m_p-m_{p+1}}\geq\sqrt{\gamma}\sum_{p=1}^{N}\|D_p-D_{p+1}\|_F,
	\end{align*}
	which implies $\sum_{p=1}^{\infty}(D_p-D_{p+1})$ is convergent if~(\ref{mj inequality}) holds. Therefore $\displaystyle{\lim_{N\to\infty}D_N}$ exists.
	Similarly we can prove $(ii)$.\\
	\noindent {\it (iii).} Note that, from~(\ref{mj equality}) we have, for $N>0$,
	\begin{align*}
	\sum_{p=1}^{N}\sqrt{m_p-m_{p+1}}\geq& \sum_{p=1}^{N}\|(X_1)_{p+1}(C_p-C_{p+1})\|_F\\
	=&\sum_{p=1}^{N}\sqrt{{\rm tr}[(C_p-C_{p+1})^T(X_1)_{p+1}^T(X_1)_{p+1}(C_p-C_{p+1})]}.
	\end{align*}
	If $X_1^*:=\displaystyle{\lim_{p\to\infty}(X_1)_{p}}$ is of full rank, it follows that, for large $p$, $(X_1)_{p+1}^T(X_1)_{p+1}\geq\eta I_k,$ for some $\eta> 0.$ Therefore, we have
	\begin{align*}
	&\sum_{p=1}^{N}\sqrt{m_p-m_{p+1}}\geq\sqrt{\eta}\sum_{p=1}^{N}\|C_p-C_{p+1}\|_F.
	\end{align*}
	Following the same argument as in the previous proof we can conclude $\displaystyle{\lim_{p\to\infty}C_p}=C^*$ exists if ~(\ref{mj inequality}) holds. Recall from~(\ref{x_p}-\ref{D_p}), we have,
	\begin{align*}\label{stationary point}
	\left\{\begin{array}{ll}
	&((X_1)_{p+1}-A_1)\odot W_1\odot W_1-(A_2-(X_1)_{p+1}C_p-B_pD_p)C_p^T=0,\\
	&(X_1)_{p+1}^T(A_2-(X_1)_{p+1}C_{p+1}-B_pD_p)=0,\\
	&(A_2-(X_1)_{p+1}C_{p+1}-B_{p+1}D_p)D_p^T=0,\\
	& B_{p+1}^T(A_2-(X_1)_{p+1}C_{p+1}-B_{p+1}D_{p+1})=0.
	\end{array}\right.
	\end{align*}
	Taking limit as $p\to\infty$ in above we have
	\begin{align*}
	\left\{\begin{array}{ll}
	\frac{\partial}{\partial X_1}F({X}_1^*,C^*,B^*,D^*)&=(X_1^*-A_1)\odot W_1\odot W_1+(B^*D^*+X_1^*C^*-A_2){C^*}^{T}=0,\\
	\frac{\partial}{\partial C}F({X}_1^*,C^*,B^*,D^*)&={{X}_1^*}^T(A_2-{X}_1^*{C}^*-B^*D^*)=0,\\
	\frac{\partial}{\partial B}F({X}_1^*,C^*,B^*,D^*)&=(A_2-{X}_1^*C^*-B^*D^*){D^*}^T=0,\\
	\frac{\partial}{\partial D}F({X}_1^*,C^*,B^*,D^*)&={B^*}^T(A_2-{X}_1^*C^*-B^*D^*)=0.
	\end{array}\right.
	\end{align*}
	Therefore, $({X}_1^*,C^*,B^*,D^*)$ is a stationary point of $F$.~This completes the proof.
	\qquad\end{proof}

\section{\bf Numerical results}
In this section we will demonstrate numerical results of our weighted rank constrained algorithm and show the convergence to the solution given by Golub,~Hoffman,~and Stewart when $W_1\to\infty$ as predicted by our theorems in Section 2. All experiments were performed on a computer with 3.1 GHz Intel Core i7-4770S processor and 8GB memory.

\subsection{\bf Experimental setup}
To perform our numerical simulations we construct two different variety of test matrix $A$. The first series of experiments were performed to demonstrate the convergence of the algorithm proposed in Section 4 and to validate the analytical result in Theorem~\ref{theorem 3}.~To this end, we performed our experiments on three full rank synthetic matrices $A$ of size $300\times300$, $500\times500$, and $700\times700$ respectively. We constructed $A$ as low rank matrix plus Gaussian noise such that $A=A_0+\alpha*E_0$, where $A_0$ is the low-rank matrix,~$E_0$ is the noise matrix, and $\alpha$ controls the noise level. We generate $A_0$ as a product of two independent full-rank matrices of size $m\times r$ whose elements are independent and identically distributed~(i.i.d.) $\mathcal{N}(0,1)$ random variables such that ${\rm r}(A_0)=r$. We generate $E_0$ as a noise matrix whose elements are i.i.d. $\mathcal{N}(0,1)$ random variables as well. In our experiments we choose $\displaystyle{\alpha=0.2\max_{i,j}(A_{ij})}$. The true rank of the test matrices are 10\% of their original size but after adding noise they become full rank.

To compare the performance of our algorithm with the existing weighted low-rank approximation algorithms, we are interested in which $A$ has a known singular value distribution. To address this, we construct $A$ of size $50\times 50$ such that ${\rm r}(A)=30$. Note that, $A$ has first 20 singular values distinct, and last~10 singular values repeated.~It is natural to consider the cases where $A$ has large and small condition number.~That is,~we demonstrate the performance of WLR in two different cases:~condition number $\kappa(A)$ of $A$:~(i)~small~and~(ii)~large, where $\kappa(A)=\frac{\sigma_{max}}{\sigma_{min}}.$

\subsection{\bf Implementation details}
Let $A_{WLR}=(X_1^*\;\; X_1^*C^*+B^*D^*)$ where $(X_1^*, C^*$, $B^*, D^*)$ be a solution to~(\ref{main problem 2}). We denote $(A_{WLR})_p$ as our approximation to $A_{WLR}$ at $p$th iteration. Recall that $(A_{WLR})_p=((X_1)_{p}\;\; (X_1)_{p}C_p+B_pD_p).$ We denote $\|(A_{WLR})_{p+1}-(A_{WLR})_{p}\|_F=Error_p$ and use $\frac{Error_p}{\|(A_{WLR})_{p}\|_F}$ as a measure of the relative error. For a threshold $\epsilon>0$ the stopping criteria of our algorithm at the $p$th iteration is  $Error_p<\epsilon$ or $\frac{Error_p}{\|(A_{WLR})_{p}\|_F}<\epsilon$ or if it reaches the maximum iteration.~The algorithm performs the best when we initialize $X_1$ and $D$ as random normal matrices and $B$ and $C$ as zero matrices.~Throughout this section we set $r$ as the target low rank and $k$ as the total number of columns we want to constrain in the observation matrix.~The algorithm takes approximately $35.9973$ seconds on an average to perform 2000 iterations on a $300\times 300$ matrix for fixed $r, k$, and $\lambda$.

\subsection{\bf Experimental results on algorithm in section 4.1}
We first verify our implementation of the algorithm for computing $A_{WLR}$ for fixed weights.~Throughout this subsection we set the target low-rank $r$ as the true rank of the test matrix and $k=0.5r$. To obtain the accurate result we run every experiment 25 times with random initialization and plot the average outcome in each case.
A threshold equal to $2.2204\times10^{-16}$ (``machine $\epsilon$'') is set for the experiments in this subsection.
For Figure 5.1, we consider a nonuniform weight with entries in $W_1$ randomly chosen from the interval $[\lambda,\zeta]$, where $\min_{\substack{1\le i\le m\\1\le j\le k}}(W_1)_{ij}=\lambda$ and $\max_{\substack{1\le i\le m\\1\le j\le k}}(W_1)_{ij}=\zeta$ in the first block $W_1$ and $W_2=\mathbbm{1}$ and plot iterations versus relative error.~Relative error is plotted in logarithmic scale along $Y$-axis.
\begin{figure}[htpb]
	\begin{subfigure}[c]{.5\textwidth}
		\centering
		\includegraphics[width=\textwidth]{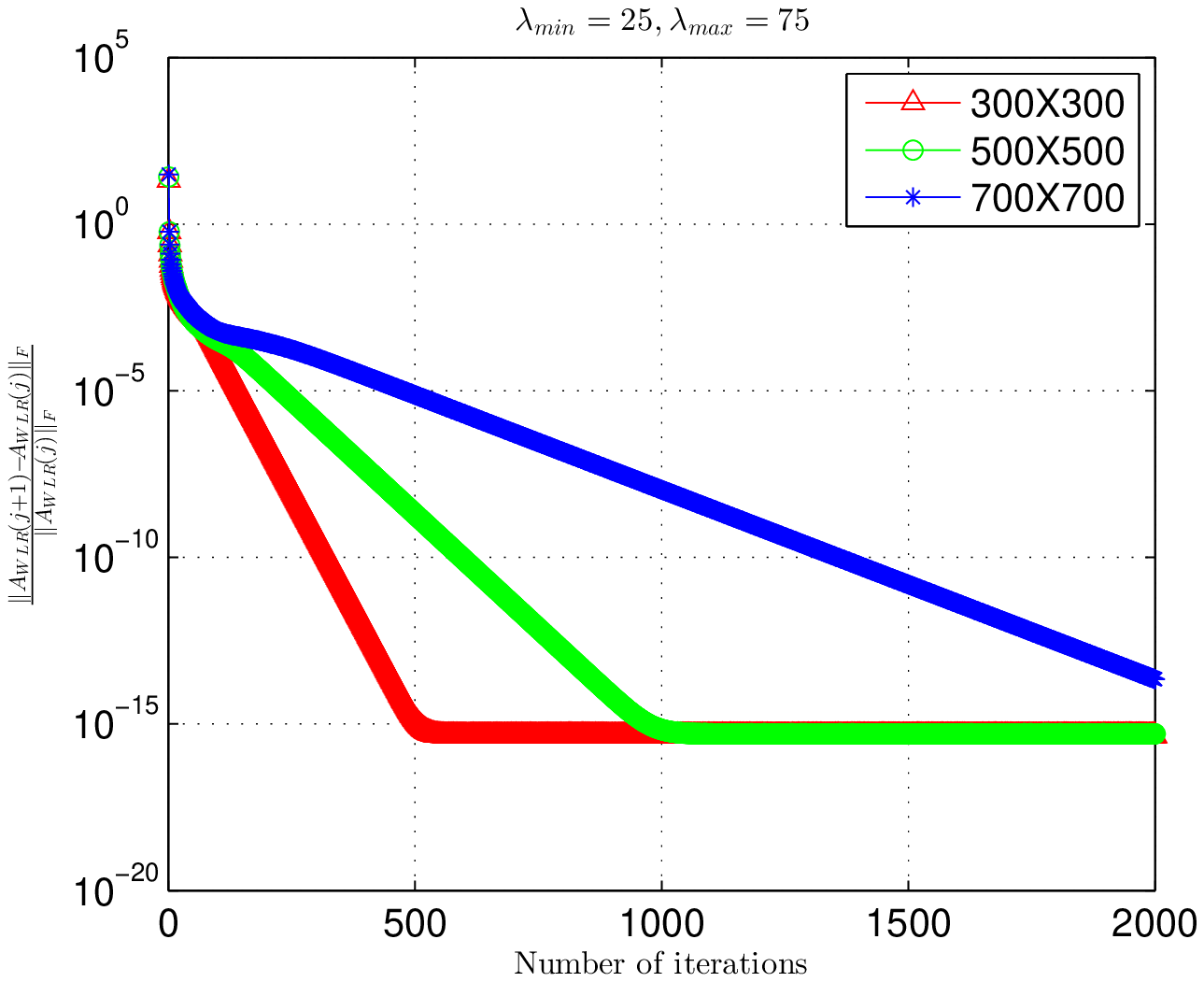}
		\caption{}
	\end{subfigure}
	\begin{subfigure}[c]{.5\textwidth}
		\centering
		\includegraphics[width=\textwidth]{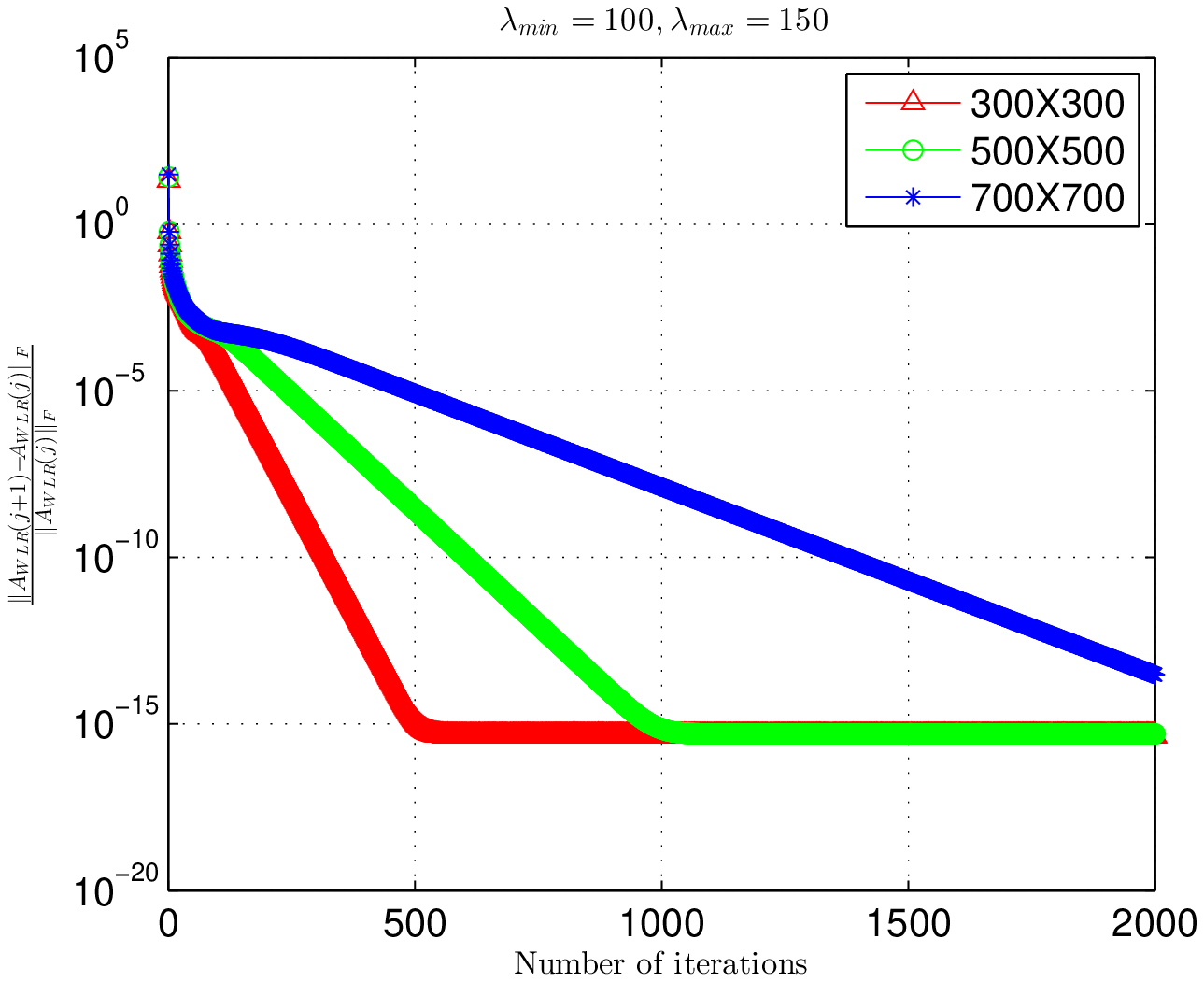}
		\caption{}
	\end{subfigure}
	\caption{Iterations vs Relative error:~(a)~$\lambda=25,\zeta=75$;~(b)~$\lambda=100,\zeta=150$.}
\end{figure}
\begin{figure}[htpb]
	\begin{subfigure}[c]{.5\textwidth}
		\centering
		\includegraphics[width=\textwidth]{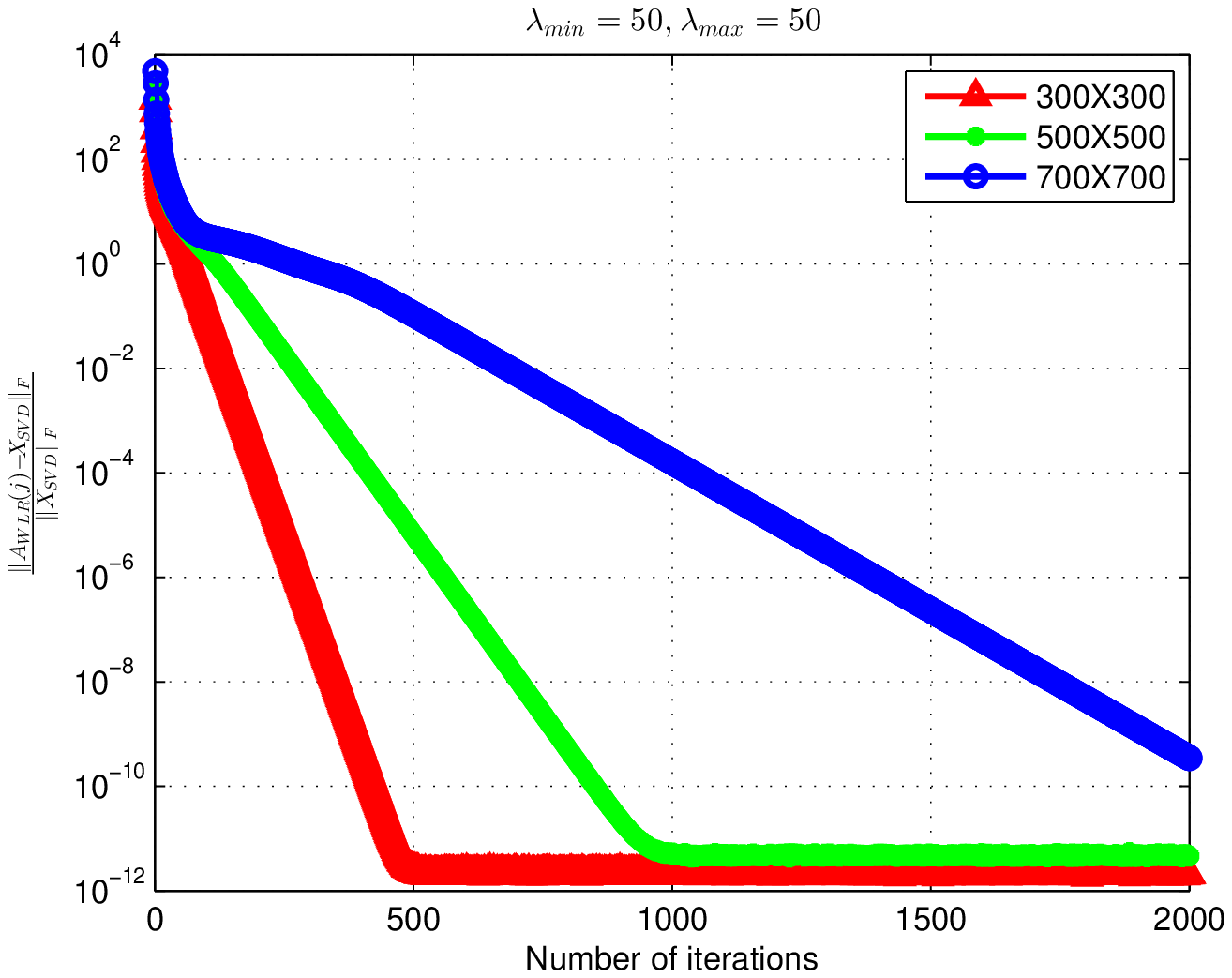}
		\caption{}
	\end{subfigure}
	\begin{subfigure}[c]{.5\textwidth}
		\centering
		\includegraphics[width=\textwidth]{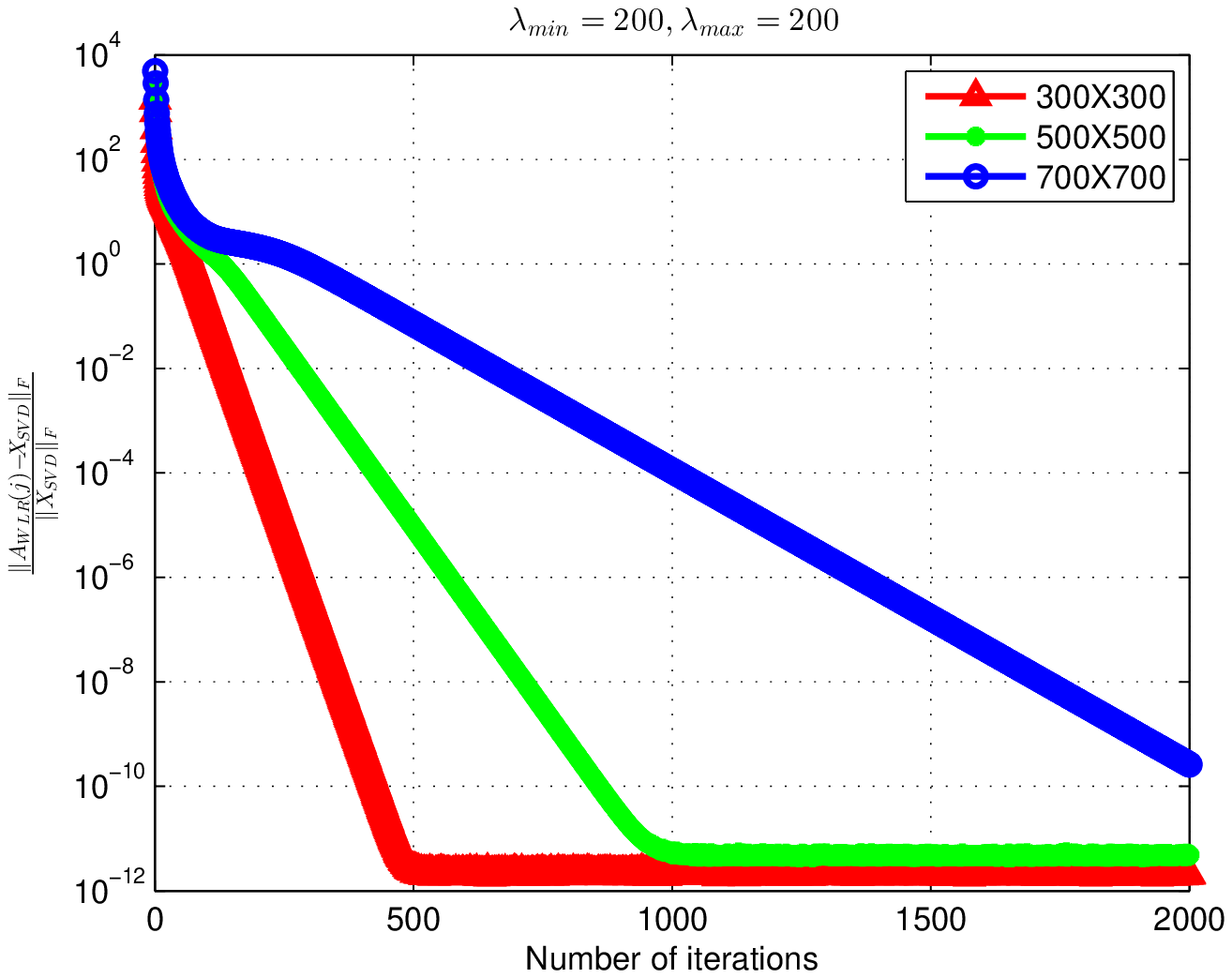}
		\caption{}
	\end{subfigure}
	\caption{Iterations vs $\frac{\|A_{WLR}(j)-X_{SVD}\|_F}{\|X_{SVD}\|_F}$:~(a)~$\lambda = 50$;~(b)~$\lambda = 200.$}
\end{figure}
Next, we consider a uniform weight in the first block $W_1$ and $W_2=\mathbbm{1}$. Recall that,~in this case the solution to problem~(\ref{hadamard problem}) can be given in closed form. That is, when $W_1=\lambda\mathbbm{1}$,~the rank $r$ solutions to~(\ref{hadamard problem}) are $X_{SVD}=[\frac{1}{\lambda}\tilde{X}_1\;\;\tilde{X}_2]$, where $[\tilde{X}_1\;\;\tilde{X}_2]$ is obtained in closed form using a SVD of $[\lambda A_1\;\;A_2]$. In Figure 5.2, we plot iterations versus $\frac{\|A_{WLR}(j)-X_{SVD}\|_F}{\|X_{SVD}\|_F}$ in logarithmic scale.
From Figures 5.1 and 5.2 it is clear that the algorithm in Section 4.1 converges. Even for the bigger size matrices the iteration count is not very high to achieve the convergence.

\subsection{Numerical results supporting Theorem~\ref{theorem 3}}		
We now demonstrate numerically the rate of convergence as stated in Theorem 2.1 when the block of weights in $W_1$ goes to $\infty$ and $W_2=\mathbbm{1}$. First we use an uniform weight $W_1=\lambda\mathbbm{1}$ and $W_2=\mathbbm{1}$. The algorithm in Section 4 is used to compute $A_{WLR}$ and SVD is used for calculating $A_G$, the solution to~(\ref{golub's problem}) when $A=(A_1\;\;A_2)$. We plot $\lambda$ vs. $\lambda\|A_G-A_{WLR}\|_F$ where $\lambda\|A_G-A_{WLR}\|_F$ is plotted in logarithmic scale along $Y$-axis. We run our algorithm 20 times with the same initialization and plot the average outcome. A threshold equal to $10^{-7}$ is set for the experiments in this subsection.~For Figure 5.3 we set $\lambda=[1:50:1000]$.
\begin{figure}[htpb]
	\centering
	\begin{subfigure}[c]{.495\textwidth}
		\centering
		\includegraphics[width=\textwidth]{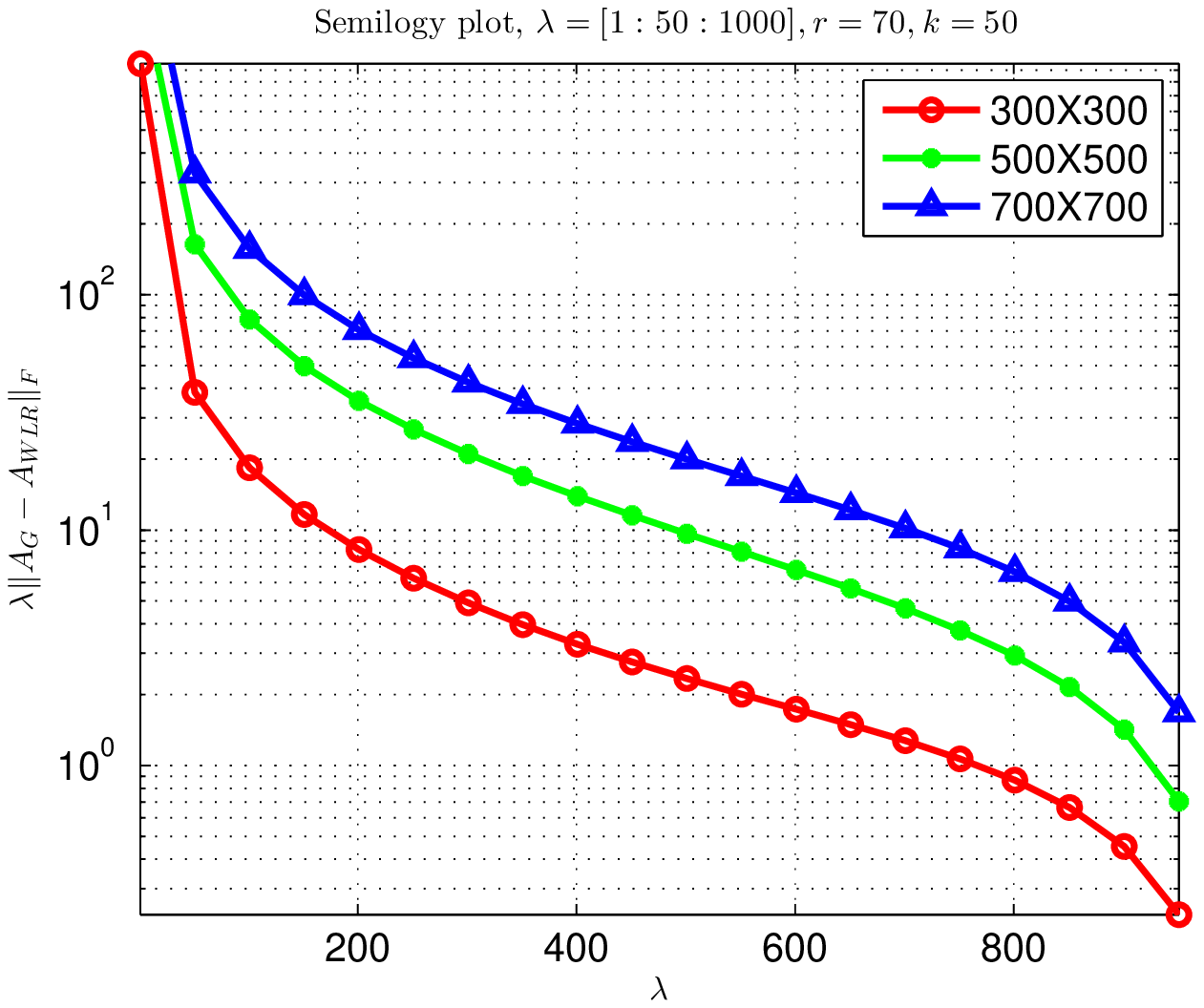}
		\caption{}
	\end{subfigure}
	\begin{subfigure}[c]{.495\textwidth}
		\centering
		\includegraphics[width=\textwidth]{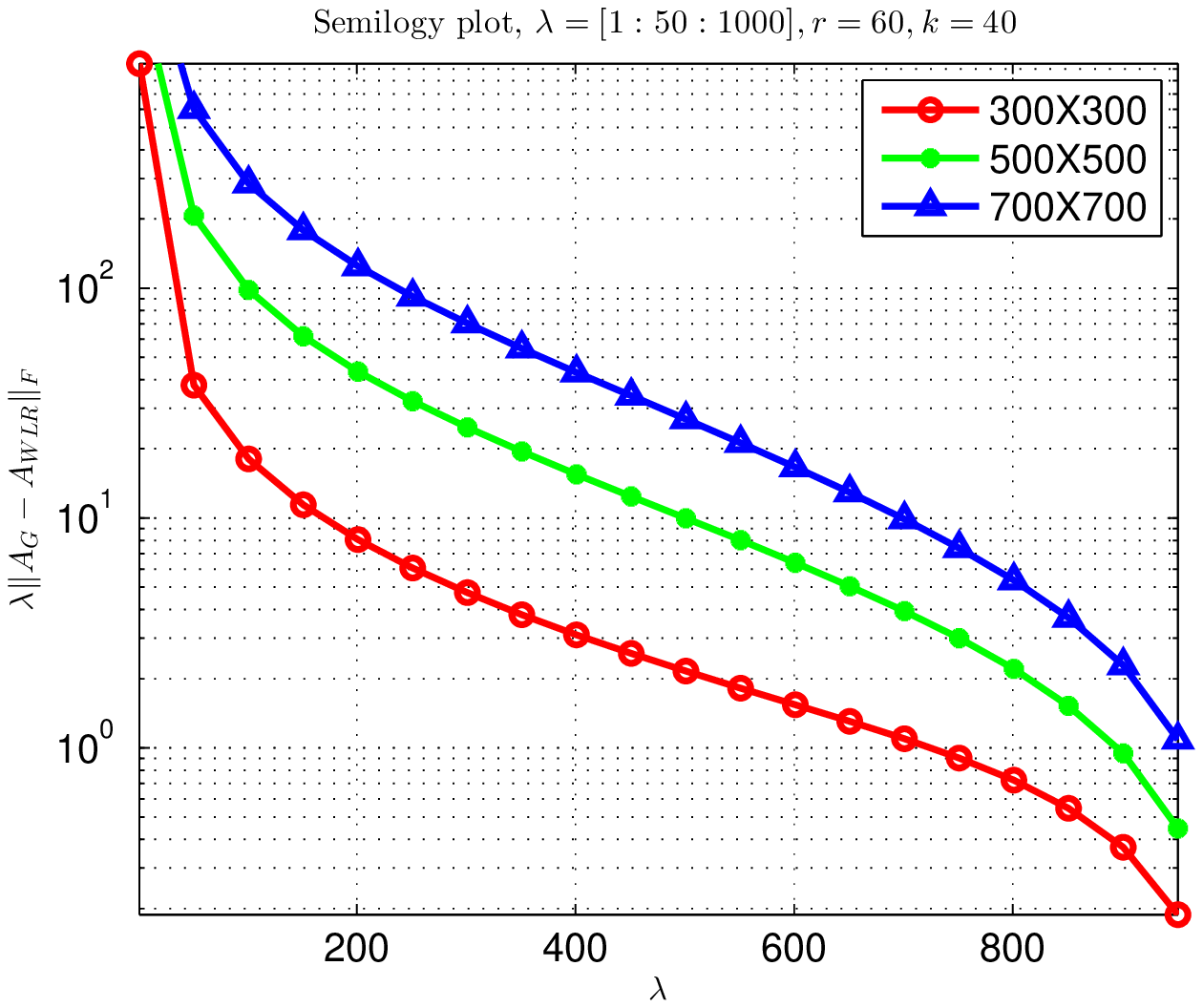}
		\caption{}
	\end{subfigure}
	\caption{$\lambda$ vs. $\lambda\|A_G-A_{WLR}\|_F$:~(a)~$(r,k)=(70,50)$,~(b)~$(r,k)=(60,40)$.}
\end{figure}
\begin{figure}[htbp]
	\centering
	\begin{subfigure}[c]{.495\textwidth}
		\centering
		\includegraphics[width=\textwidth]{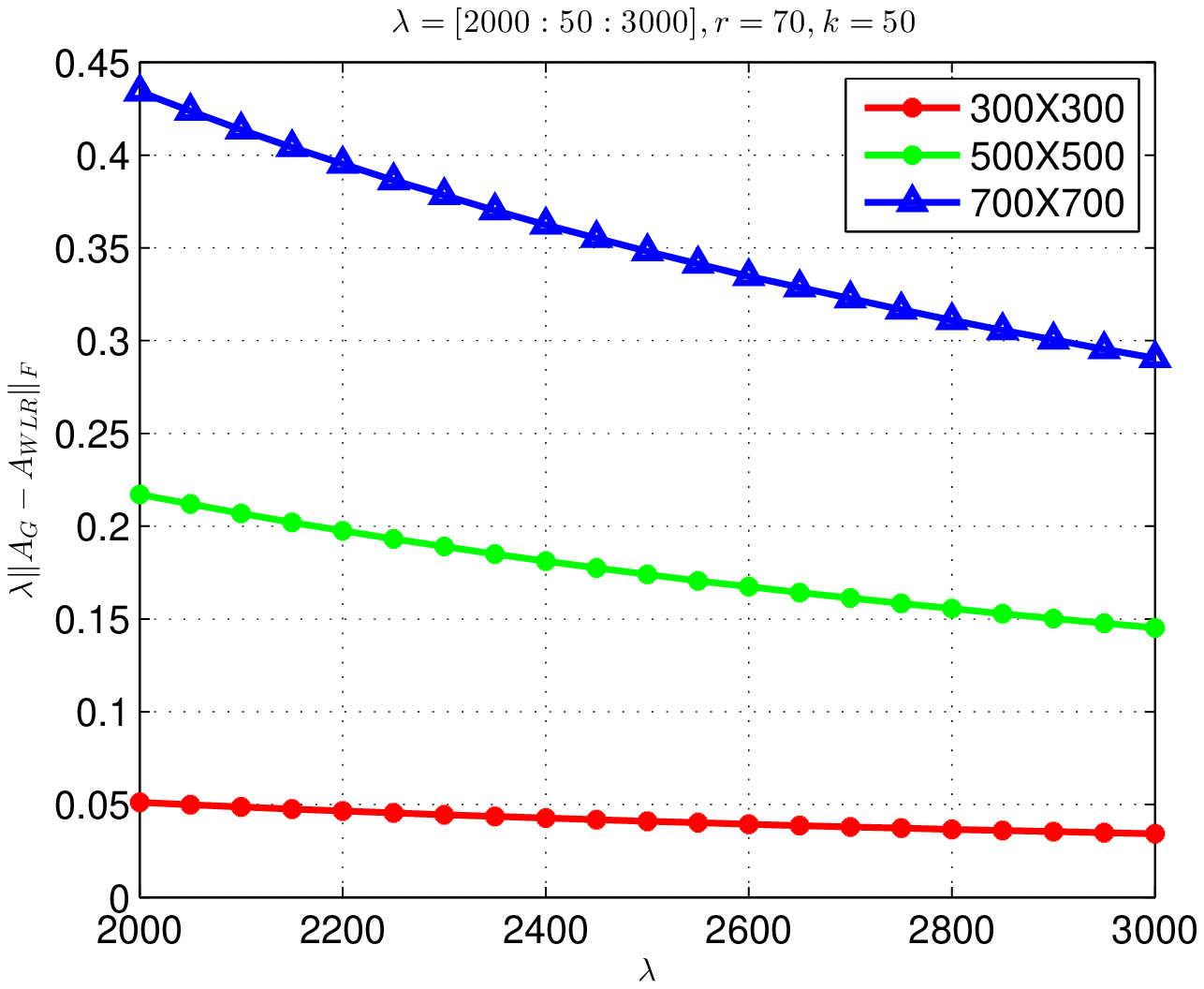}
		\caption{}
	\end{subfigure}
	\begin{subfigure}[c]{.495\textwidth}
		\centering
		\includegraphics[width=\textwidth]{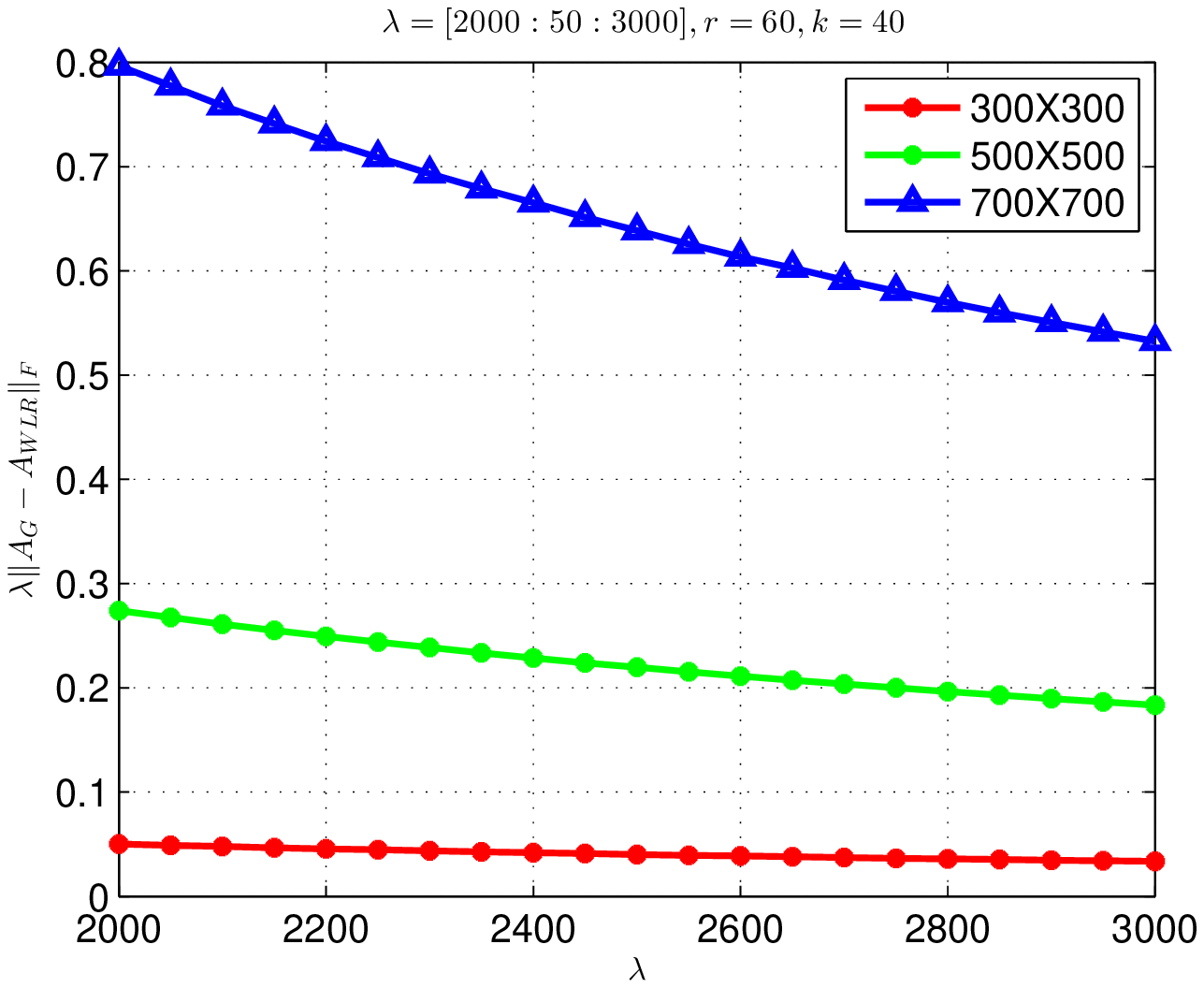}
		\caption{}
	\end{subfigure}
	\caption{$\lambda$ vs. $\lambda\|A_G-A_{WLR}\|_F$:~(a)~$(r,k)=(70,50)$,~(b)~$(r,k)=(60,40)$.}
\end{figure}
%~For Figure~5.3~(a) we consider $r=70, k=50$ and for Figure~5.3~(b) we choose $r=60, k=40$.
The plots indicate for an uniform $\lambda$ in $W_1$ the convergence rate is at least $O(\frac{1}{\lambda}), \lambda\to\infty.$  Next we consider a nonuniform weight in the first block $W_1$ and $W_2=\mathbbm{1}$.~We consider~$(W_1)_{ij}\in[2000,2020], [2050,2070]$ and so on. For Figure 5.4, $\lambda\|A_G-A_{WLR}\|_F$~(recall $\min_{\substack{1\le i\le m\\1\le j\le k}}(W_1)_{ij}=\lambda$) is plotted in regular scale along $Y$-axis. The curves in figure~5.4 are not always strictly decreasing but it is encouraging to see that they stay bounded. Figures~5.3~and~5.4~provide numerical evidence in supporting Theorem~\ref{theorem 3}. As established in Theorem~\ref{theorem 3} the above plots demonstrate the convergence rate is at least $O(\frac{1}{\lambda}), \lambda\to\infty.$

\subsection{Comparison with state-of-the-art general weighted algorithms}
In this section, we compare the performance of our special weighted algorithm on synthetic data with the standard weighted total alternating least squares~(WTALS) method proposed in~\cite{markovosky,manton} and the expectation maximization~(EM) method proposed by~Srebro and Jaakkola~\cite{srebro}.~The existing algorithms are for general weighted case but for our purpose we consider partial weighting in them.~Additionally,~we compare the performance of our algorithm with the standard alternating least squares,~WTALS,~and the EM method~\cite{srebro,markovosky} for $k=0$ case.
For the numerical experiments in this section, we are interested to see how the distribution of the singular values affects the performance of our algorithm compare to other state-of-the-art algorithms.

\subsubsection{Performance compare to other weighted low-rank approximation algorithms}
\begin{figure}[htpb]
	\centering
	\begin{subfigure}[c]{.495\textwidth}
		\centering
		\includegraphics[width=\textwidth]{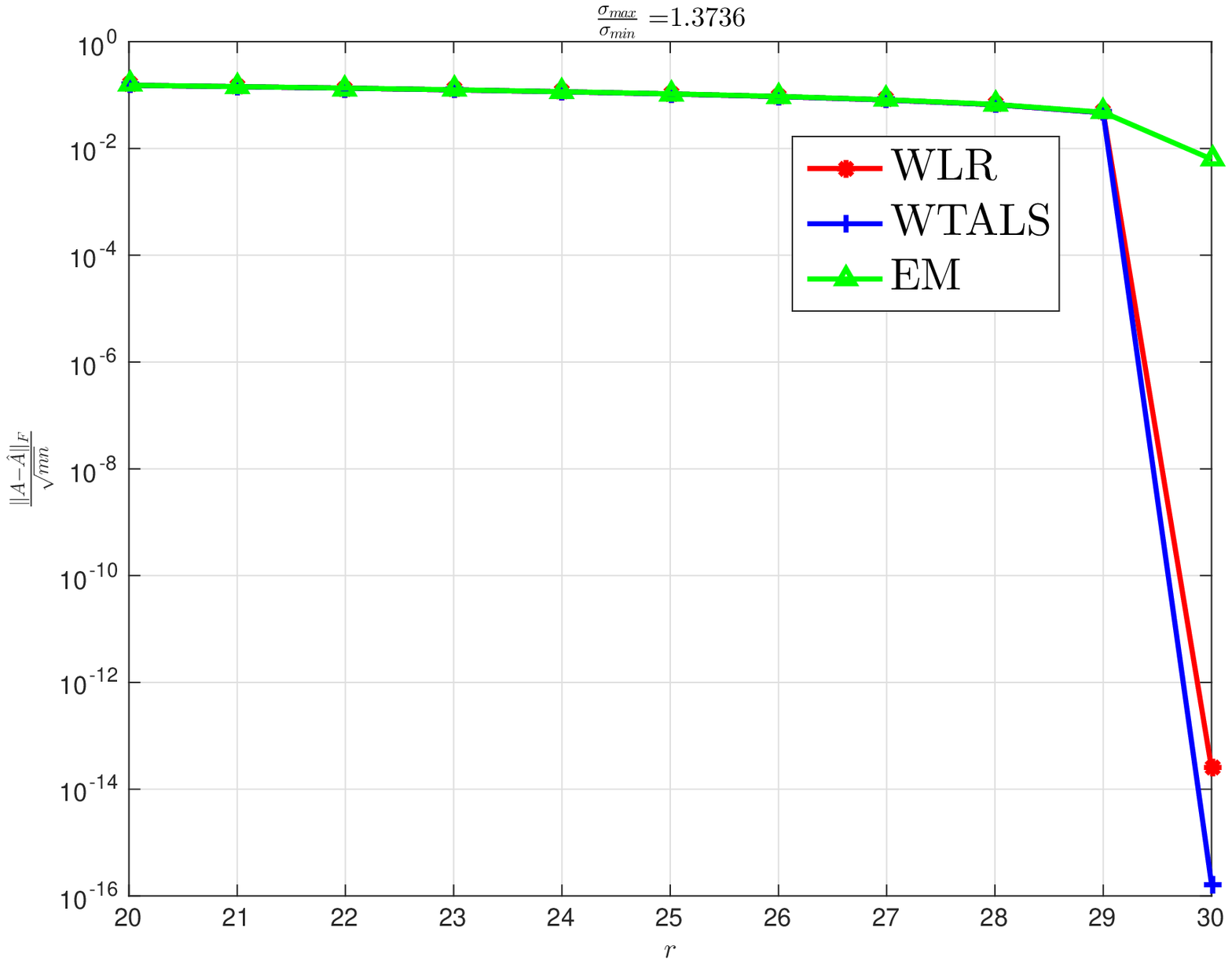}
		\caption{}
	\end{subfigure}
	\begin{subfigure}[c]{.495\textwidth}
		\centering
		\includegraphics[width=\textwidth]{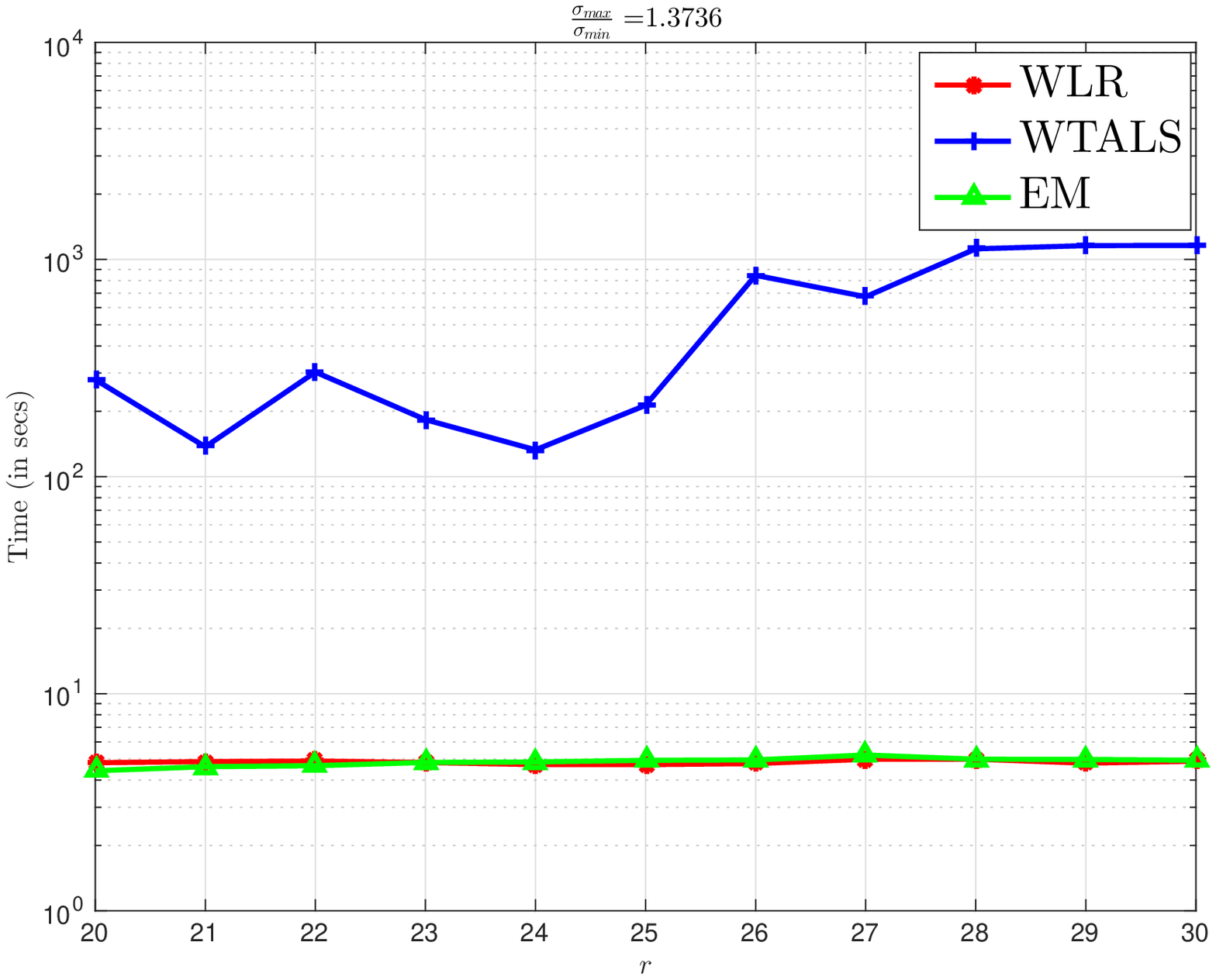}
		\caption{}
	\end{subfigure}
	\caption{Comparison of WLR with other methods when $\kappa(A)$ small:~(a)~$r$ versus RMSE,~(b)~$r$ versus time. We have $\kappa(A)=1.3736$, $r=[20:1:30]$, $k=10$, and $(W_1)_{ij}\in[50,1000]$.}
\end{figure}

The weights in the first block are chosen randomly from a large interval. We set $(W_1)_{ij}\in[50,1000]$ and $W_2=\mathbbm{1}$.~For WTALS, as specified in the software package, we consider \verb+ max_iter = 1000, threshold = 1e-10+~\cite{markovosky}.~For EM,~we choose~\verb+ max_iter = 5000, threshold = 1e-10+,~and for WLR, we set \verb+ max_iter+ \verb+= 2500, threshold = 1e-16+. For the performance measure, we use the standard root mean square error~(RMSE) which is $\|A-\hat{A}\|_F/\sqrt{mn}$, where $\hat{A}\in\mathbb{R}^{m\times n}$ is the low-rank approximation of $A$ obtained by using different weighted low-rank approximation algorithm.~The \verb+MATLAB+ code for the EM method is written by the authors following the algorithm proposed in~\cite{srebro}.~For computational time of WLR and EM, the authors do not claim the optimized performance of their codes.~However, the initialization of $X$ plays a crucial role in promoting convergence of the EM method to a global, or a local minimum, as well as the speed with which convergence is attained.
\begin{figure}[H]
	\centering
	\begin{subfigure}[c]{.495\textwidth}
		\centering
		\includegraphics[width=\textwidth]{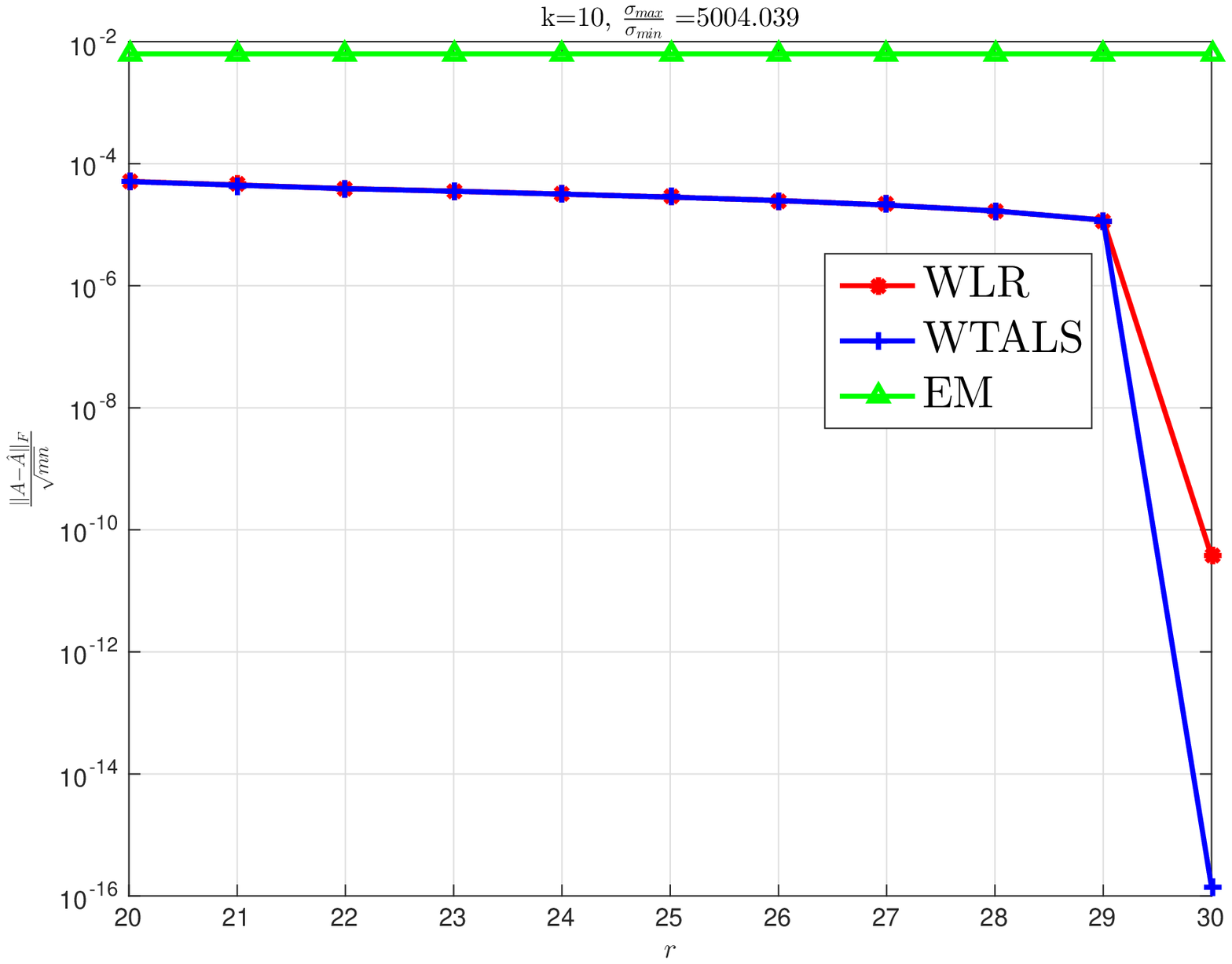}
		\caption{}
	\end{subfigure}
	\begin{subfigure}[c]{.495\textwidth}
		\centering
		\includegraphics[width=\textwidth]{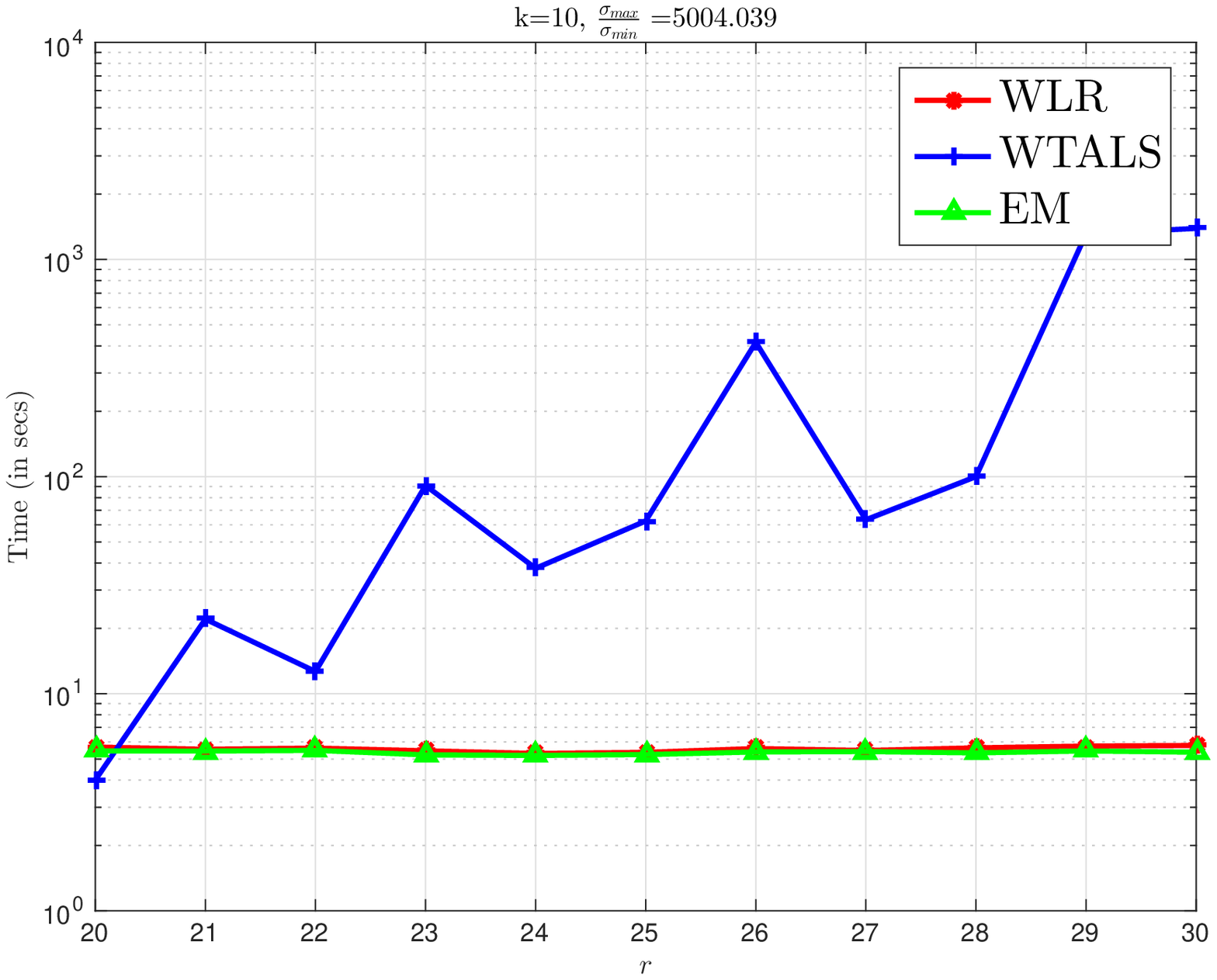}
		\caption{}
	\end{subfigure}
	\caption{Comparison of WLR with other methods when $\kappa(A)$ large:~(a)~$r$ versus RMSE,~(b)~$r$ versus time. We have $\kappa(A)= 5.004\times10^3$, $r=[20:1:30]$,  $k=10$, and $(W_1)_{ij}\in[50,1000]$.}
\end{figure}
\begin{figure}[H]
	\centering
	\begin{subfigure}[c]{.495\textwidth}
		\centering
		\includegraphics[width=\textwidth]{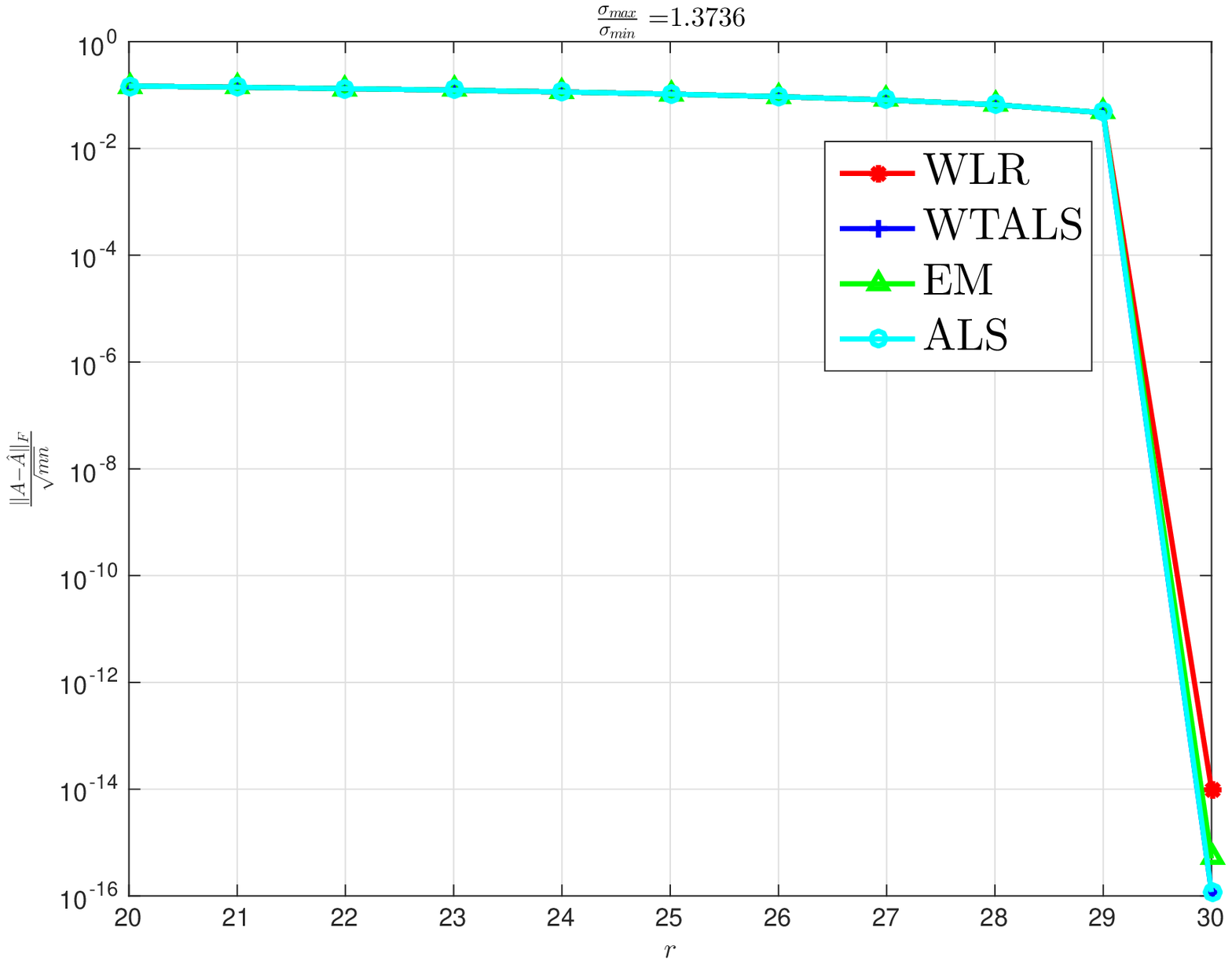}
		\caption{}
	\end{subfigure}
	\begin{subfigure}[c]{.495\textwidth}
		\centering
		\includegraphics[width=\textwidth]{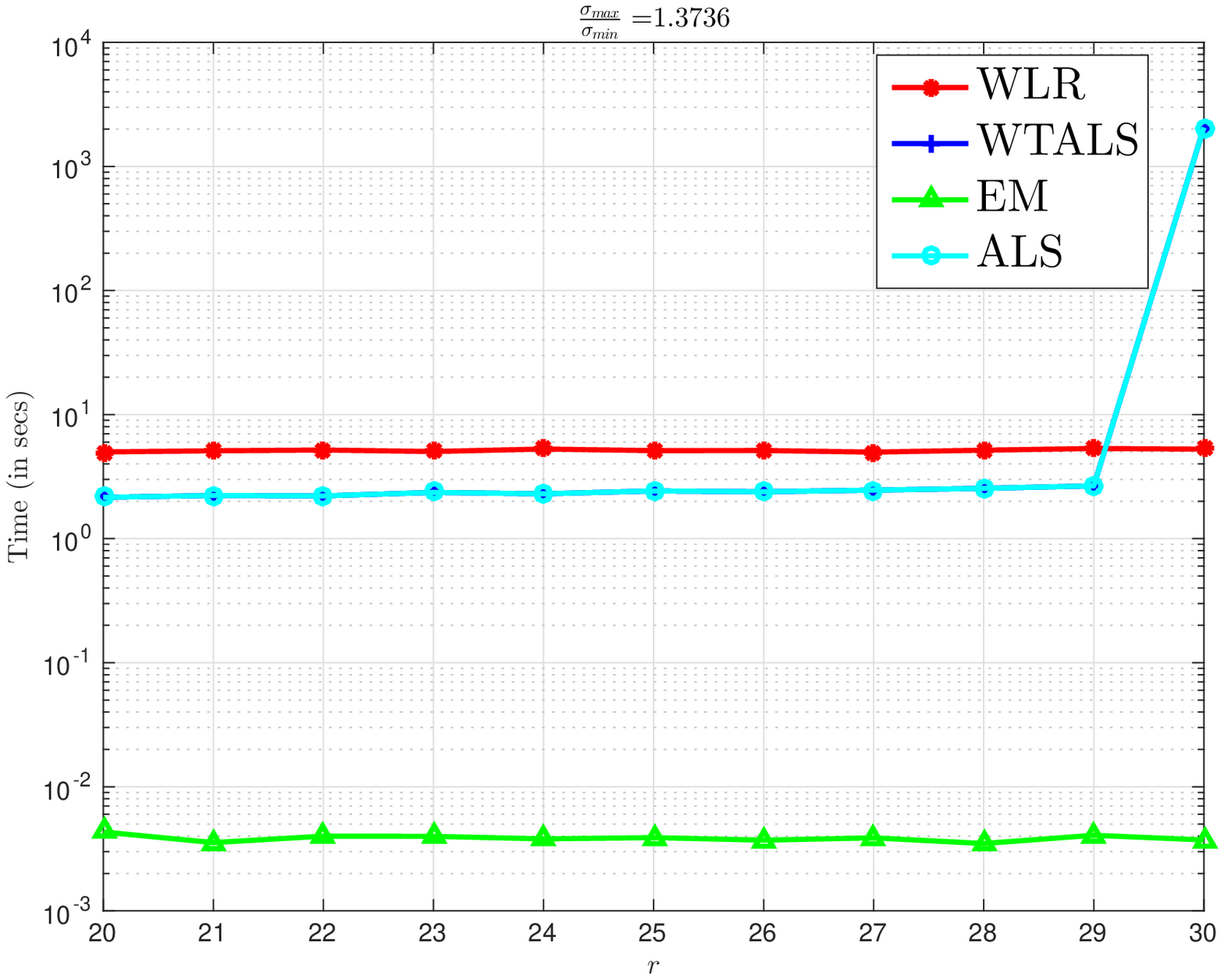}
		\caption{}
	\end{subfigure}
	\caption{Comparison of WLR with other methods for $k=0$ and $\kappa(A)$ small:~(a)~$r$ versus RMSE,~(b)~$r$ versus time. We have $\kappa(A)=1.3736$, $r=[20:1:30]$.}
\end{figure}
To implement the EM method, as mentioned in~\cite{srebro}, first we rescale the weight matrix to $W_{EM}=\frac{1}{\max_{ij}(W_1)_{ij}}(W_1\;\;\mathbbm{1})$. For a given threshold of weight bound~$\epsilon_{EM}$, we initialize $X$ to a zero matrix if $\min_{ij}(W_{EM})_{ij}\le \epsilon_{W_{EM}}$, otherwise we initialize $X$ to $A$.~Initialization for WLR is same as specified in Section~5.2.~To obtain the accurate result we run each experiment 10 times and plot the average outcome.~Both RMSE and computational time are plotted in logarithmic scale along $Y$-axis.~Figure 5.5 and 5.6 indicate that WLR is more efficient in handling bigger size matrices than  WTALS~\cite{markovosky} with the comparable performance measure.~This can be attributed by the fact that WTALS uses a weight matrix of size $mn\times mn$ for the given input size $m\times n$, which is both memory and time inefficient.~On the other hand, Figure~5.5 and 5.6 support the fact that as mentioned in~\cite{srebro}, the EM method is computationally effective, however in some cases might converge to a local minimum instead of global.

\subsubsection{Performance comparison for $k=0$~(Alternating Least Squares)}
For $k=0$ we set the weight matrix as $W=\mathbbm{1}$ for all weighted low-rank approximation algorithm.~Moreover, we include the classic alternating least squares algorithm to compare between the accuracy of the methods.~As specified in Section 5.5.1, the stopping criterion for all weighted low-rank algorithms are kept the same and RMSE is used for performance measure.~We run each experiment 10 times and plot the average outcome.~Figure 5.7 and 5.8 indicate that WLR has comparable performance.~However,~the standard ALS, WTALS, and EM method is more efficient than WLR, as for $W=\mathbbm{1}$ case, each method uses SVD to compute the solution.
\begin{figure}[H]
	\centering
	\begin{subfigure}[c]{.495\textwidth}
		\centering
		\includegraphics[width=\textwidth]{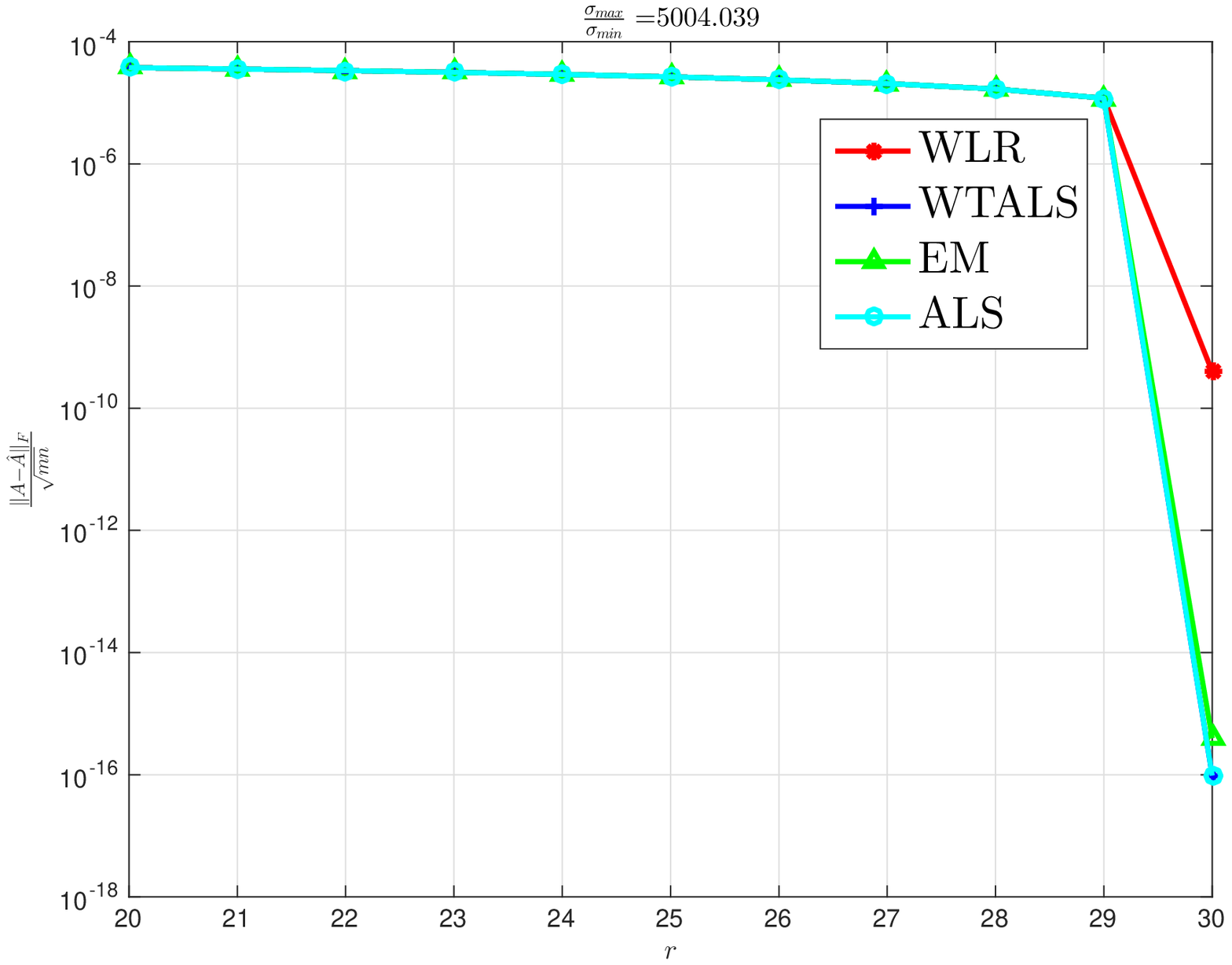}
		\caption{}
	\end{subfigure}
	\begin{subfigure}[c]{.495\textwidth}
		\centering
		\includegraphics[width=\textwidth]{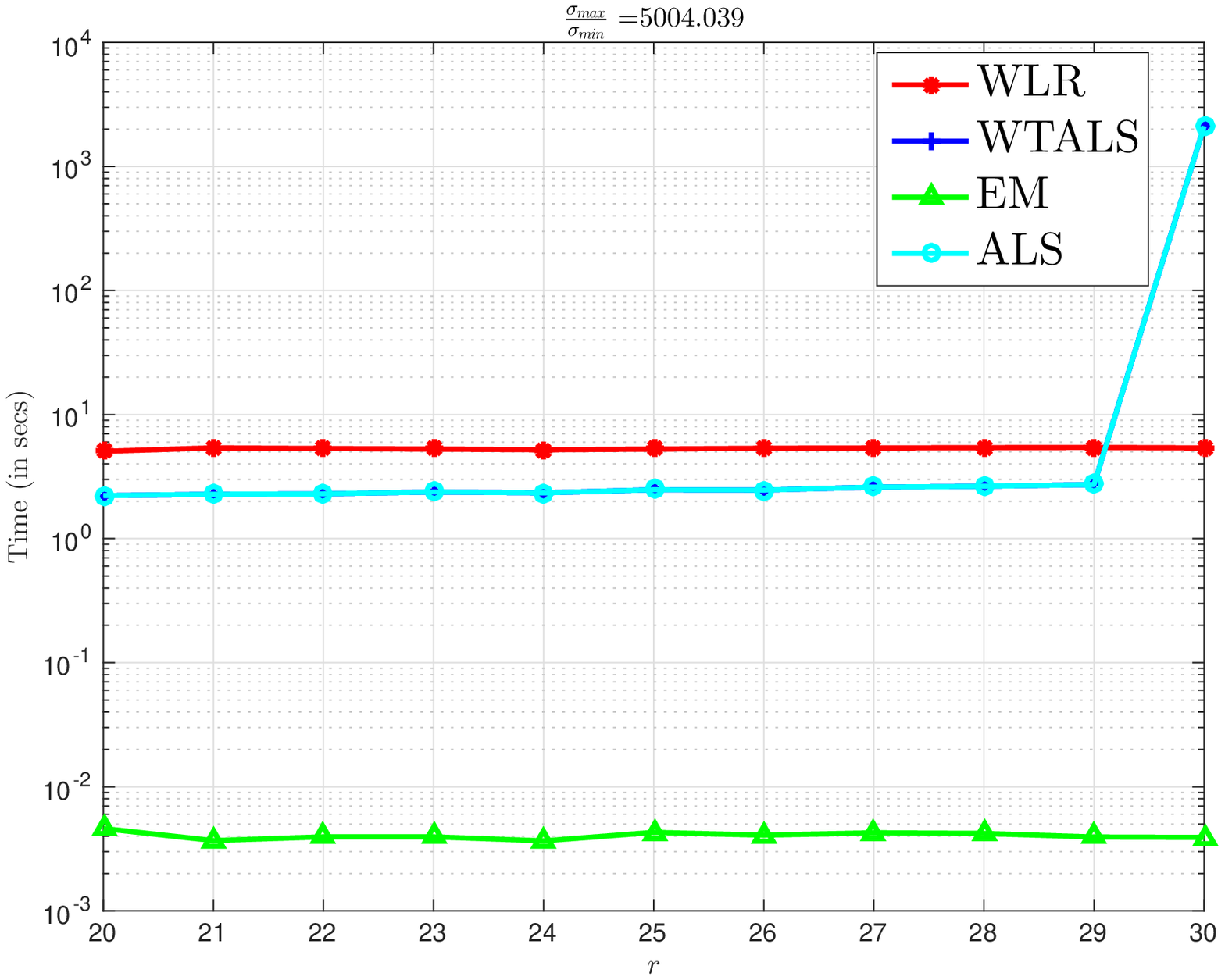}
		\caption{}
	\end{subfigure}
	\caption{Comparison of WLR with other methods for $k=0$ and $\kappa(A)$ large:~(a)~$r$ versus RMSE,~(b)~$r$ versus time. We have $\kappa(A)=5.004\times10^3$, $r=[20:1:30]$.}
\end{figure}
\subsubsection{Performance compare to constrained low-rank approximation of Golub-Hoffman-Stewart}
As mentioned in our analytical results, one can expect, with appropriate conditions, the solutions to~(\ref{hadamard problem}) will converge and the limit is $A_G$, the solution to the constrained low-rank approximation problem by Golub-Hoffman-Stewart. We now show the effectiveness of our special weighted algorithm compare to other state-of-the-art weighted low rank algorithms when $(W_1)_{ij}\to \infty,$ and $W_2=\mathbbm{1}$.~In this section, the weights in $W_1$ are chosen to be large to show effectiveness of our algorithm for large weighted case in the first block.~SVD is used for calculating $A_G$, the solution to~(\ref{golub's problem}), when $A=(A_1\;\;A_2)$, for varying $r$ and fixed $k$. Considering $A_G$ as the true solution we use the RMSE measure $\|A_G-\hat{A}\|_F/\sqrt{mn}$ as the performance metric, where $\hat{A}\in\mathbb{R}^{m\times n}$ is the low-rank approximation of $A$ obtained by different weighted low-rank algorithm. From Figure 5.9 it is evident that WLR has the superior performance compare to the other state-of-the-art general weighted low-rank approximation algorithms when $W_1\to\infty$.
\begin{figure}[H]
	\centering
	\begin{subfigure}[c]{.495\textwidth}
		\centering
		\includegraphics[width=\textwidth]{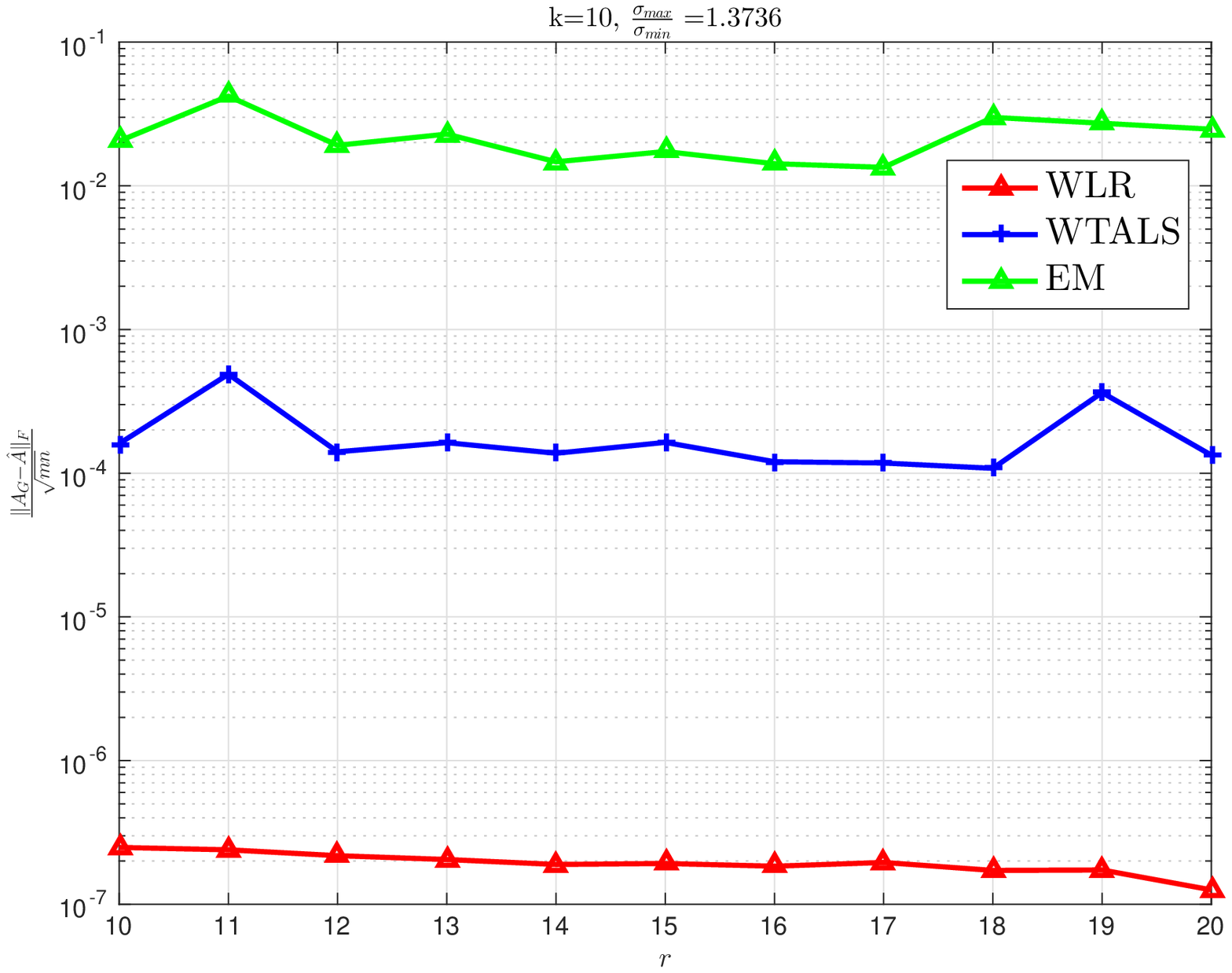}
		\caption{}
	\end{subfigure}
	\begin{subfigure}[c]{.495\textwidth}
		\centering
		\includegraphics[width=\textwidth]{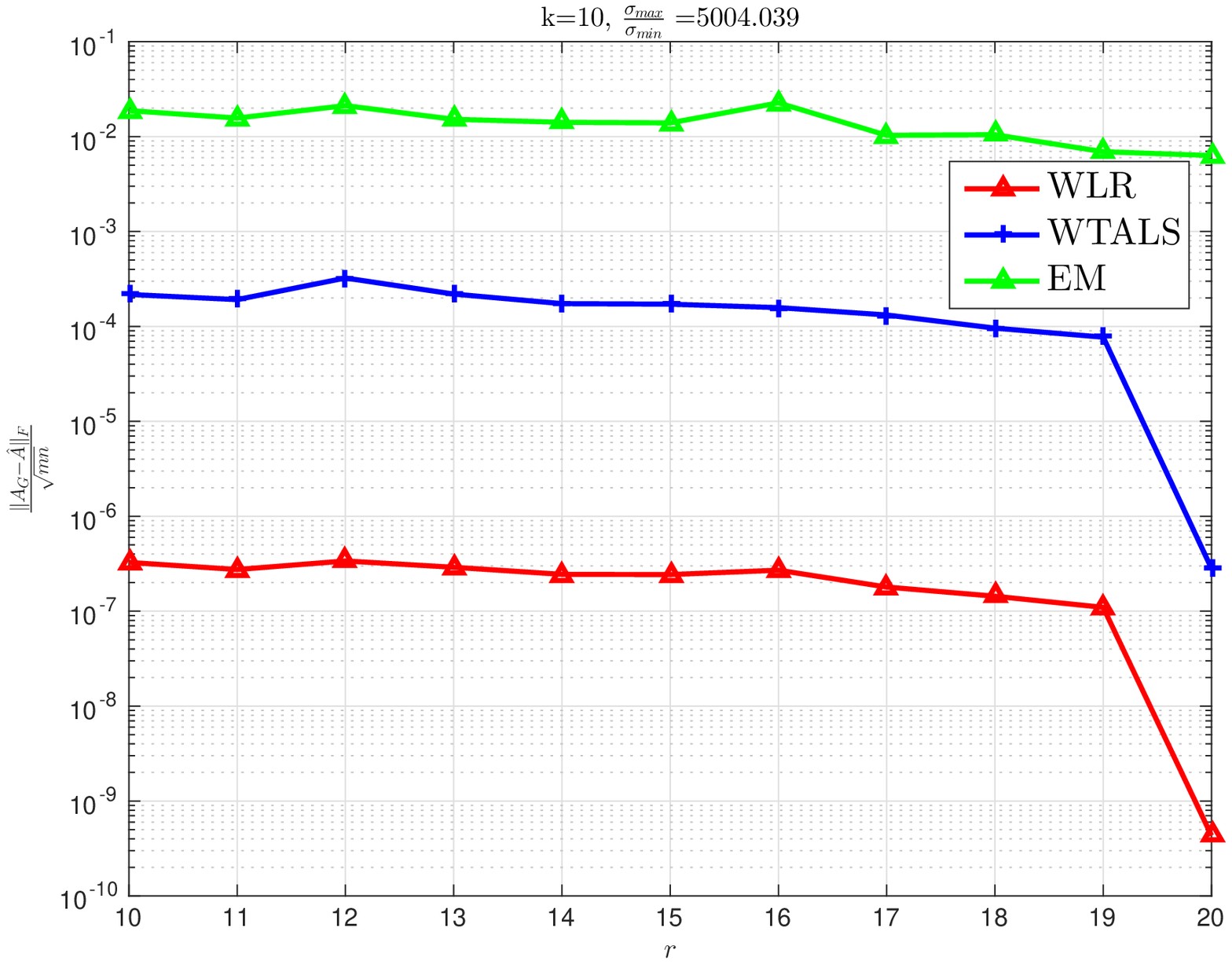}
		\caption{}
	\end{subfigure}
	\caption{$r$ vs $\|A_G-\hat{A}\|_F/\sqrt{mn}$ for different methods, $(W_1)_{ij}\in[500,1000], W_2=\mathbbm{1}$, $r=10:1:20$, and $k=10$:~(a)~$\kappa(A)$ small,~(b)~$\kappa(A)$ large.}
\end{figure}
To conclude, WLR has comparable or superior performance compare to the existing general weighted low-rank approximation algorithms for the special case of weight with fairly less computational time. Even when the columns of the given matrix are not constrained, that is $k=0$, its performance is comparable to the standard ALS. Additionally, WLR and EM method can easily handle bigger size matrices and easier to implement for real world problems.~On the other hand, WTALS requires more computational time and is not memory efficient to handle large scale data~(see table 5.1).~Another important feature of our algorithm is that it does not assume any particular condition about the matrix $A$ and performs equally well in every occasion.

\begin{table}[H]
	\caption{Average computation time~(in seconds) for each algorithm to converge to $A_G$}
	\begin{center}
		\begin{tabular}{|c|c|c|c|}
			\hline $\kappa(A)$ & WLR & EM & WTALS  \\
			\hline 1.3736  & 6.5351 & 6.1454  & 205.1575 \\
			\hline  $5.004\times10^3$ & 8.8271 & 8.1073&107.0353 \\
			\hline
		\end{tabular}
	\end{center}
\end{table}

\section*{Acknowledgments}
We would like to thank the anonymous referees and the associate editor Dr. Ivan Markovsky for providing many useful references, and for their valuable comments and suggestions which improved the presentation and results of this paper. We would also like to thank Dr.~Afshin Dehghan, at the Center for Research in Computer Vision, University of Central Florida for his invaluable comments in the numerical results section.		
%\newpage

\end{document}